\documentclass{amsart}

\usepackage{geometry}
\numberwithin{equation}{section}
\numberwithin{figure}{section}

\usepackage{amsmath, amsthm, amssymb, amstext, amsfonts}
\usepackage{enumerate}
\usepackage{moreenum}
\usepackage{graphicx}
\usepackage[T1]{fontenc} %needed for polic ogonek
\usepackage[applemac]{inputenc}

\usepackage{graphicx}
\usepackage[dvipsnames]{xcolor}
\usepackage{tikz}
\usetikzlibrary{arrows,patterns}
\usepackage[bookmarks,hyperfootnotes=false, psdextra=true]{hyperref}
\usepackage[nameinlink, capitalise, noabbrev]{cleveref}
\colorlet{darkishRed}{red!60!black}
\colorlet{darkishBlue}{blue!60!black}
\colorlet{darkishGreen}{green!50!black}
\colorlet{lightishGreen}{green!70!black}
\hypersetup{
    draft = false,
    bookmarksopen=true,
    colorlinks,
    linkcolor={darkishBlue},
    citecolor={lightishGreen},
    urlcolor={darkishBlue}
}

\theoremstyle{plain}

\newtheorem{thm}{Theorem}[section]
\crefname{thm}{Theorem}{Theorems}
\newtheorem{lem}[thm]{Lemma}
\crefname{lem}{Lemma}{Lemmas}
\newtheorem{cor}[thm]{Corollary}
\newtheorem{prop}[thm]{Proposition}
\newtheorem{obs}[thm]{Observation}

\newtheorem{prob}[thm]{Problem}
\newtheorem*{prob*}{Problem}
\newtheorem{conj}[thm]{Conjecture}

\newtheorem{mainresult}{Theorem}
\newtheorem{maincorollary}[mainresult]{Corollary}
\crefname{maincorollary}{Corollary}{Corollaries}

\crefname{mainresult}{Theorem}{Theorems}

\theoremstyle{definition}

\newtheorem{ex}[thm]{Example}

\newcommand{\N}{\ensuremath{\mathbb{N}}}
\newcommand{\Z}{\ensuremath{\mathbb{Z}}}

\newcommand{\cE}{\ensuremath{\mathcal{E}}}

\newcommand{\cR}{\ensuremath{\mathcal{R}}}

\newcommand{\cS}{\ensuremath{\mathcal{S}}}

\newcommand{\cV}{\ensuremath{\mathcal{V}}}

\newcommand{\sm}{\ensuremath{\smallsetminus}}

\newcommand{\sub}{\subseteq}

\def\qt{quasi-tran\-si\-tive}
\def\qi{quasi-iso\-metric}

\def\qg{quasi-geo\-de\-sic}

\def\lf{locally finite}

\newcommand{\comment}[1]{}

\def\?#1{\vadjust{\vbox to 0pt{\vss\vskip-8pt\leftline{%
     \llap{\hbox{\vbox{\pretolerance=-1
     \doublehyphendemerits=0\finalhyphendemerits=0
     \hsize16truemm\tolerance=10000\small
     \lineskip=0pt\lineskiplimit=0pt
     \rightskip=0pt plus16truemm\baselineskip8pt\noindent
     \hskip0pt        %(without this, the first word is never hyphenated!)
     #1\endgraf}\hskip7truemm}}}\vss}}}

\newenvironment{txteq*}
  {
    \begin{equation*}
    \begin{minipage}[c]{0.85\textwidth} % set width to 0.9 x textwidth
    \em                                % switch on emph
  }
  {\end{minipage}\end{equation*}\ignorespacesafterend}

\let\eps=\varepsilon
\let\phi=\varphi

\newcommand{\defn}[1]{{\color{darkishRed}{\emph{#1}}}}
\newcommand{\defnm}[1]{{\color{darkishRed}{#1}}}

\usepackage{enumitem}

\newtheorem{innercustomthm}{}

\newtheorem{innercustomcor}{}
\newenvironment{customthm}[1]
  {\renewcommand\theinnercustomthm{#1}\innercustomthm}
  {\endinnercustomthm}

\newenvironment{customcor}[1]
  {\renewcommand\theinnercustomcor{#1}\innercustomcor}
  {\endinnercustomcor}

\newcounter{claimcounter}[thm]
\newtheorem*{claim*}{Claim}
\crefname{claim}{Claim}{Claims}

\usepackage[skip=6pt]{subcaption} % mehrere Abbildungen nebeneinander/übereinander

\makeatletter
\def\namedlabel#1#2{\begingroup
   \def\@currentlabel{#2}%
   \label{#1}\endgroup
}
\makeatother

\usepackage{etoolbox}

\newbool{arXiv}
\boolfalse{arXiv} % Set the boolean variable to false

\newcommand{\ArXivOrNot}[2]{\ifbool{arXiv}{{#1}}{{#2}}}

\usepackage{csquotes}
\usepackage[backend=biber, style=alphabetic, sorting=nyt, maxnames=99, maxalphanames=99, giveninits=true, sortcites=true]{biblatex}

\makeatletter
\newcommand\thankssymb[1]{\textsuperscript{\@fnsymbol{#1}}}
\makeatother

\newcommand{\COMMENT}[1]{{}}

\begin{document}

\title[A coarse Halin Grid Theorem]{A coarse Halin Grid Theorem with applications to quasi-transitive, locally finite graphs}
\author[Sandra Albrechtsen]{Sandra Albrechtsen}
\author[Matthias Hamann]{Matthias Hamann\thankssymb{1}}
\thanks{\thankssymb{1} Funded by the Deutsche Forschungsgemeinschaft (DFG) - Project No.\ 549406527.}

\email{sandra@albrechtsen-mail.de, matthias.hamann@uni-hamburg.de}
\address{University of Hamburg, Department of Mathematics, Bundesstrasse 55 (Geomatikum), 20146 Hamburg, Germany}


\begin{abstract}
We prove a coarse version of Halin's Grid Theorem: Every one-ended, locally finite graph that contains the disjoint union of infinitely many rays as an asymptotic minor also contains the half-grid as an asymptotic minor. 
More generally, we show that the same holds for arbitrary (not necessarily one-ended or locally finite) graphs under additional, necessary assumptions on the minor-models of the infinite rays.
This resolves a conjecture of Georgakopoulos and Papasoglu.

As an application, we show that every one-ended, quasi-transitive, locally finite graph contains the half-grid as an asymptotic minor and as a diverging minor. This in particular includes all locally finite Cayley graphs of one-ended finitely generated groups and solves a problem of Georgakopoulos and Papasoglu. 
\end{abstract}

\maketitle

\section{Introduction}

An \defn{end} of a graph~$G$ is an equivalence class of rays in~$G$ where two rays are equivalent if there are infinitely many pairwise disjoint paths between them in~$G$. An end is \defn{thick} if it contains infinitely many pairwise disjoint rays. The \defn{half-grid} (i.e.\ the grid\footnote{In a grid two vertices $(m,n)$ and $(m',n')$ are adjacent if and only if $|m - m'| + |n - n'| = 1$} on $\N\times\N$), is an obvious example: it has infinitely many disjoint rays, and they all belong to the same end, which hence must be thick.

One of the cornerstones of infinite graph theory is the following result of Halin \cite[Satz 4$^\prime$]{halin65}, which asserts that this observation has a converse, and thus the half-grid can be thought of as a prototype for thick ends: 

\begin{customthm}{Halin's Grid Theorem}
\namedlabel{HalinsGridTheorem}{Halin's Grid Theorem}
    \emph{Every graph with a thick end contains the half-grid as a minor.}    
\end{customthm}

However, the half-grid minor which \ref{HalinsGridTheorem} provides can have arbitrary distortion compared to the geometry of the ambient graph. Consider for example the infinite clique~$K_{\aleph_0}$. Its unique end is thick, and it clearly contains the half-grid as a minor. However, $K_{\aleph_0}$ is quasi-isometric\footnote{Two graphs $G, H$ are \defn{quasi-isometric} if there exists a function $f\colon V(H) \rightarrow V(G)$ and $M \geq 1$ and $A \geq 0$ such that $M^{-1}\cdot d_H(u,v)-A \leq d_G(f(u), f(v)) \leq M\cdot d_H(u,v)+A$ for all $u,v \in V(H)$ and $d_G(f(V(H), w) \leq A$ for all $w \in V(G)$.} to a point, so~$K_{\aleph_0}$ inherits none of the geometric properties of the half-grid. In contexts where the geometry matters it would be desirable to have a version of \ref{HalinsGridTheorem} which preserves some of the geometry of the half-grid.

Let us first consider Cayley graphs, a class of graphs for which this problem is particularly interesting. In geometric group theory, where one is interested in the large-scale geometry of (the Cayley graph of) a group $\Gamma$, one wants the analysis of $\Gamma$ to be stable under quasi-isometries and, in particular, independent of the choice of a finite generating system. Since having a thick end is a quasi-isometric invariant for Cayley graphs, it is natural to ask whether Cayley graphs with a thick end contain a minor-model of the half-grid that survives under quasi-isometries.

For this, we consider `fat' and `asymptotic' minors, a special kind of minors that play an important role in the new area of `coarse graph theory' (see \cite{GP2023+} for an overview).
A graph $X$ is a \defn{$K$-fat} minor of a graph~$G$, for some $K \in \N$, if there is a minor-model of~$X$ in~$G$ whose branch sets and branch paths are pairwise at least $K$ apart, except that we do not require this for incident branch set-path pairs (cf.\ \cref{subsec:fatminors}). The graph~$X$ is an \defn{asymptotic minor} of~$G$ if~$X$ is a $K$-fat minor of $G$ for every $K \in \N$.

In contrast to the usual minors, asymptotic minors are preserved under quasi-isometries~\cite{GP2023+}. In particular, if a graph contains the half-grid as an asymptotic minor, then its large-scale geometry will still resemble that of a graph with a thick end.
Moreover, it does not depend on the choice of a finite generating set whether a \lf\ Cayley graph of a finitely generated group contains a fixed graph as an asymptotic minor.

A natural question that arises is whether we can improve \cref{HalinsGridTheorem} for Cayley graphs so that it yields an asymptotic half-grid minor.
This question was already posed by Georgakopoulos and Papasoglu, who asked this for Cayley graphs of one-ended, finitely generated groups. For this, note that by a result of Thomassen \cite[Proposition 5.6]{T1992}, the unique end of a one-ended, \lf\ Cayley graph must be thick. 

\begin{prob}[\rm{\cite[Problem 7.3]{GP2023+}}] \label{prob:HG:QT}
    Let $G$ be a locally finite Cayley graph of a one-ended, finitely generated group. Must the half-grid be an asymptotic minor of~$G$?
\end{prob}
 
Our main result resolves \cref{prob:HG:QT}. 
In fact, instead of constructing a $K$-fat minor-model of the half-grid for every $K \in \N$, we find a single minor-model of the half-grid which is \defn{ultra-fat}: for every $K \in \N$ its `submodel' corresponding to the subgraph of the half-grid on the vertex set $\N^2\sm [K]^2$ is $K$-fat.

\begin{mainresult} \label{main:HG:QT:OneEnded}
    Every one-ended, \qt, \lf\ graph~$G$ contains the half-grid as an ultra-fat minor.
    In particular, the half-grid is an asymptotic minor and a diverging minor of~$G$.
\end{mainresult}

\noindent A graph $G$ is \defn{\qt} if the automorphism group of~$G$ acts on $V(G)$ with only finitely many orbits.
Note that every Cayley graph of a finitely generated group is quasi-transitive.

\cref{main:HG:QT:OneEnded} is also related to problems posed by Georgakopoulos and the second author~\cite{GH24+}, who asked for diverging minors of grids. A minor-model of a graph~$X$ in a graph~$G$ \defn{diverges}, if for every two sequences $(x_n)_{n \in \N}$ and $(y_n)_{n\in \N}$ of vertices and/or edges of~$X$ such that $d_X(x_n,y_n) \rightarrow \infty$, the distances in~$G$ between their corresponding branch sets or paths tend to infinity. It is easy to see that every ultra-fat minor-model of the half-grid diverges (see \cref{prop:UFHGImpliesDivHG}).

We remark that \cref{main:HG:QT:OneEnded} lies at the frontier of what is still true: We cannot ask for coarse embeddings\footnote{See \cref{subsec:NonElliptic} for the definition.} of the (hexagonal\footnote{See \cref{sec:Halin:Grid} for the definition.}) half-grid as e.g.\ the Cayley-graph of the lamplighter group, which is one-ended, does not contain a coarsely embedded (hexagonal) half-grid (cf.\ \cite[Discussion preceding Problem~4.2]{GH24+}).
\medskip

What about \qt, \lf\ graphs that have more than one end? 
A \qt\ graph is \emph{accessible} if there exists $n\in\N$ such that every two of its ends can be `separated' by at most~$n$ vertices.
By a standard argument about accessible, quasi-transitive graphs, we obtain that we may relax the condition on~$G$ in \cref{main:HG:QT:OneEnded} from `one-ended' to `accessible with a thick end' (see \cref{cor:HalfGrid:AccessibleGraph}).
Since every connected, \qt\ graph without thick ends is quasi-isometric to a tree by Kr\"on and M\"oller \cite[Theorem 5.5]{KM08}, this implies the following corollary:

\begin{maincorollary} \label{main:HalfGrid:QuasiIsoToTree}
    Every accessible, connected, \qt, locally finite graph is quasi-isometric to a tree if and only if it does not contain the half-grid as an ultra-fat minor.
\end{maincorollary}

\cref{main:HG:QT:OneEnded} has another application.
Two rays $S,R$ in a graph $G$ \defn{diverge} if for every $K \in \N$ there exist tails $S', R'$ of $S, R$, respectively, such that $d_G(S', R') \geq K$. 
As the half-grid contains infinitely many pairwise disjoint rays, \cref{main:HG:QT:OneEnded} implies the following: 

\begin{maincorollary}\label{main:DivergingRays}
    Every \qt, \lf\ graph that contains infinitely many pairwise disjoint rays also contains infinitely many pairwise diverging rays.
\end{maincorollary}

\noindent This solves \cite[Conjecture 4.3]{GH24+} for quasi-transitive graphs.
\medskip

Let us now leave the realm of Cayley graphs and quasi-transitive graphs and consider arbitrary infinite graphs with thick ends. As mentioned earlier, not every graph with a thick end contains the half-grid as an asymptotic minor, and hence we can in general not hope to replace `minor' with `asymptotic minor' in \ref{HalinsGridTheorem}.
Indeed, the already mentioned $K_{\aleph_0}$ is an obvious example, but there are also locally finite graphs that have a thick end but that do not contain the half-grid as an asymptotic minor.\footnote{Let $G$ be the graph obtained from the ray $r_1r_2 \dots$ by replacing each vertex $r_i$ with a clique $K_i$ of size~$i$ and each edge $r_ir_{i+1}$ by a complete bipartite graph between $K_i$ and $K_{i+1}$. Then $G$ is locally finite by definition, and it is easy to check that $G$ has a thick end, but the half-grid is not even a $2$-fat minor of $G$.}
However, there is another formulation of \ref{HalinsGridTheorem}, which we will prove to have a coarse counterpart. For this, consider the following equivalent way of stating \ref{HalinsGridTheorem} for one-ended graphs: \emph{Let~$G$ be a one-ended graph. If $\omega\cdot R$ is a minor of~$G$, then the half-grid is a minor of~$G$.} Here, \defn{$\omega\cdot R$} denotes the disjoint union of countably many rays.

Since having $\omega\cdot R$ as an asymptotic minor is a necessary condition for having the half-grid as an asymptotic minor, but graphs with thick ends need in general not have $\omega\cdot R$ as an asymptotic minor, a more natural way to transfer \ref{HalinsGridTheorem} into the coarse world is the following, as proposed by Georgakopoulos and Papasoglu:

\begin{prob}[\rm{\cite[Problem~7.2]{GP2023+}}] \label{prob:HalinGrid}
    Let $G$ be a one-ended graph. If $\omega \cdot R$ is an asymptotic minor of~$G$, must the half-grid be an asymptotic minor of~$G$?
\end{prob}

The answer to this question is in the negative for general graphs, as we will show by presenting a counterexample in \cref{sec:CounterEx}.
However, we answer this question positively for \lf\ graphs.

\begin{mainresult} \label{main:Halin:HalfGrid}
    Let $G$ be a one-ended, locally finite graph. If $\omega \cdot R$ is an asymptotic minor of $G$, then the half-grid is an asymptotic minor of $G$.
\end{mainresult}

\noindent We remark that we also prove a second version of \cref{main:Halin:HalfGrid} (see \cref{thm:Halin:HalfGrid}~\ref{itm:Halin:HG:UF:Copy}) where we impose a stronger condition on the minor-model of $\omega\cdot R$, and in return obtain the half-grid as an ultra-fat minor. This result will be an important tool for our proof of \cref{main:HG:QT:OneEnded}.
\medskip

In fact, we will prove a more general version of \cref{main:Halin:HalfGrid} for arbitrary infinite (not necessarily one-ended or locally finite) graphs (see \cref{thm:Halin:HalfGrid:NotLocFin,thm:Halin:HalfGrid:UltraFat}), where we impose two further conditions on the rays in the fat minor-models of $\omega \cdot R$.
To cover graphs with more than one end, we assume that the rays in the fat minor-models of $\omega\cdot R$ are equivalent.
Further, to cover graphs with vertices of infinite degree, we assume that no two distinct rays of the fat minor-model of $\omega\cdot R$ can be separated by removing finitely many balls of finite radius.
While the first condition is clearly necessary, we will show in \cref{ex:Halin:HalfGrid:NotLocFin} that also the second condition is necessary.

\subsection{More related work}

Our work complements existing work about coarse graph theory on infinite (quasi-transitive) graphs \cite{GP2023+,AH24,M24+,L25,EGG25,LMPRQ23}. 
It is also related to existing work on characterising \qt, \lf\ graphs with thick ends, e.g.~\cite{KM08,GH24+,AH24,A11}.

In \cite{AH24} we showed that every \qt, \lf\ graph with a thick end whose cycle space is generated by cycles of bounded length\footnote{See \cite[Section~2.3]{AH24} or \cite[Section~1.9]{Bibel} for the definition. In particular, Cayley graphs of finitely presented groups have this property.} contains the full-grid (i.e.\ the grid on $\Z\times\Z$) as an asymptotic minor and as a diverging minor. Since every such graph is accessible, \cref{main:HG:QT:OneEnded} (more precisely \cref{main:HalfGrid:QuasiIsoToTree}) covers a larger class of graphs, but only yields an asymptotic and diverging minor of the half-grid. We remark that our proof techniques in this paper are mostly independent of \cite{AH24}.

Lee, Mart\'inez-Pedroza and Rodr\'iguez-Quinche \cite{LMPRQ23} introduced a coarse version of the classical cops and robbers game. 
Using the above mentioned result in~\cite{AH24}, Esperet, Gahlawat and Giocanti \cite{EGG25} proved for \lf\ Cayley graphs of finitely presented groups that if one cop cannot catch the robber in the coarse game, then infinitely many cops are needed.
Replacing in their proof the result in~\cite{AH24} by \cref{main:HG:QT:OneEnded}, this yields the same for the larger class of accessible, finitely generated groups, a result that has directly been obtained by Lehner~\cite{L25}.

Our proof of \cref{main:Halin:HalfGrid} relies on a strengthening of \cref{HalinsGridTheorem}, which was proved by Kurkofka, Melcher and Pitz in \cite{KMP2022}.

\subsection{How this paper is organised}

This paper is structured as follows.
In \cref{sec:Prelim} we recall some important definitions.
We prove \cref{main:Halin:HalfGrid} in \cref{sec:Halin:Grid} and discuss some examples regarding coarse versions of \ref{HalinsGridTheorem} in \cref{sec:CounterEx}.
In \cref{sec:QT} we deduce \cref{main:HG:QT:OneEnded} from \cref{thm:Halin:HalfGrid}, our second version of \cref{main:Halin:HalfGrid}, and in \cref{sec:ProofsOfMainResults} we prove \cref{main:DivergingRays,main:HalfGrid:QuasiIsoToTree}.

\section{Preliminaries}\label{sec:Prelim}

Our notions mainly follow~\cite{Bibel}; in particular, $\defnm{[n]} := \{1, \dots, n\}$ for $n \in \N$. In what follows, we recap some definitions which we need later.
We denote the half-grid by \defn{$HG$}.

\subsection{Paths, rays and ends}

For two sets $A,B$ of vertices of a graph~$G$, an \defn{$A$--$B$ path} is a path which meets~$A$ precisely in its first vertex and~$B$ precisely in its last vertex.

A \defn{ray} is a one-way infinite path, and a \defn{double ray} is a two-way infinite path.
A \defn{tail} of a (double) ray~$R$ is any ray $S \subseteq R$.
If $R = r_0 r_1 \dots$ is a ray, then we denote by \defn{$r_iRr_j$} for $i, j \in \N$ the subpath $r_i \dots r_j$ of~$R$, and by \defn{$r_iR$} or \defn{$R_{\geq i}$} the tail $r_i r_{i+1} \dots$ of~$R$.
Further, we denote by \defn{$Rr_i$} or \defn{$R_{\leq i}$} the subpath $r_0 \dots r_i$ of~$R$.
We use these notions analogously for double rays; in particular, if $R = \dots r_{-1}r_0r_1 \dots$ is a double ray, then \defn{$Rr_i$} and \defn{$R_{\leq i}$} denote the tail $r_ir_{i-1} \dots$ of~$R$. 

A (possibly infinite) set $U \subseteq V(G)$ \defn{separates} two rays $R,S$ in $G$ if $U$ separates $V(R)$ and $V(S)$ in~$G$, i.e.\ if $G-U$ does not contain an $R$--$S$ path.

An \defn{end} $\eps$ of a graph~$G$ is an equivalence class of rays in~$G$ where two rays are equivalent if they are joined by infinitely many disjoint paths in~$G$ or, equivalently, if for every finite set~$U \subseteq V(G)$ both rays have tails in the same component of~$G - U$.
If a ray lies in~$\eps$, then we call it an \defn{$\eps$-ray}.
An end~$\eps$ is \defn{thick} if there are infinitely many pairwise disjoint $\eps$-rays \cite{halin65}.

\subsection{Distance and balls}

Let $G$ be a graph.
We write~\defn{$d_G(v, u)$} for the distance of the two vertices~$v$ and~$u$ in~$G$. 
For two sets~$U$ and~$U'$ of vertices of~$G$, we write~\defn{$d_G(U, U')$} for the minimum distance of two elements of~$U$ and~$U'$, respectively.
If one of~$U$ or~$U'$ is just a singleton, then we omit the braces, writing $d_G(v, U') := d_G(\{v\}, U')$ for $v \in V(G)$.
If $X$ is a subgraph of $G$, then we abbreviate $d_G(U,V(X))$ as $d_G(U,X)$.

Given a set~$U$ of vertices of~$G$, the \defn{ball (in~$G$) around~$U$ of radius $r \in \N$}, denoted by~\defn{$B_G(U, r)$}, is the set of all vertices in~$G$ of distance at most~$r$ from~$U$ in~$G$.
If~$U = \{v\}$ for some~$v \in V(G)$, then we omit the braces, writing~$B_G(v, r)$ for the ball (in $G$) around~$v$ of radius~$r$.
Additionally, we abbreviate the induced subgraph on $B_G(U,r)$ of $G$ with $\defnm{G[U,r]} := G[B_G(U,r)]$.
If $X$ is a subgraph of $G$, then we abbreviate $B_G(V(X),r)$ and $G[V(X),r]$ as $B_G(X,r)$ and $G[X,r]$, respectively.

\subsection{Fat, asymptotic and diverging minors} \label{subsec:fatminors}

Let $G, X$ be graphs.
A \defn{(minor-)model} $(\cV,\cE)$ of~$X$ in~$G$ is a collection $\cV$ of disjoint sets $V_x \subseteq V(G)$ for vertices~$x$ of~$X$ such that each $G[V_x]$ is connected, and a collection~$\cE$ of internally disjoint $V_{x_0}$--$V_{x_1}$ paths $E_{e}$ for edges $e=x_0x_1$ of~$X$ such that each $E_{e}$ is disjoint from every~$V_x$ with $x \notin e$.
The $V_x$ are its \defn{branch sets} and the $E_{e}$ are its \defn{branch paths}.

Let $K \in \N$.
A model $(\cV, \cE)$ of $X$ in~$G$ is \defn{$K$-fat} if $d_G(Y,Z) \geq K$ for every two distinct $Y,Z \in \cV \cup \cE$ unless $Y = E_e$ and $Z = V_x$ for some vertex $x \in V(X)$ incident to $e \in E(X)$, or vice versa.
The graph~$X$ is a \defn{$K$-fat minor} of~$G$, denoted by \defn{$X \prec_K G$}, if~$G$ contains a $K$-fat model of~$X$. Moreover,~$X$ is an \defn{asymptotic minor} of~$G$, denoted by \defn{$X \prec_\infty G$}, if~$X$ is a $K$-fat minor of~$G$ for all $K \in \N$.

Let $\eps$ be an end of $G$. If $X$ is a one-ended graph, and $(\cV, \cE)$ is a model of $X$ in $G$, then we say that \defn{all rays of $(\cV, \cE)$ are $\eps$-rays} if every ray in $\bigcup_{x \in V(X)} G[V_x] \cup \bigcup_{e \in E(X)} E_e$ is an $\eps$-ray. 
We write \defn{$X \prec^\eps_K G$} if $G$ contains a $K$-fat model $(\cV, \cE)$ of $X$ whose rays are all $\eps$-rays, and \defn{$X \prec_\infty^\eps G$} if $X \prec^\eps_K G$ for all $K \in \N$.

Let $(\cV, \cE)$ be a model of a graph $X$ in a graph $G$.
Then $(\cV, \cE)$ \defn{diverges} (in~$G$) if for every two sequences $(x_n)_{n \in \N}$ and $(y_n)_{n\in \N}$ of vertices and/or edges of $X$ such that $d_X(x_n,y_n) \rightarrow \infty$, we have $d_G(U_n, W_n) \rightarrow \infty$ where $U_n := V_{x_n}$ if $x_n \in V(X)$ and $U_n := V(E_{x_n})$ if $x_n \in E(X)$ and analogously $W_n := V_{y_n}$ or $W_n := V(E_{y_n})$.

\subsection{Ultra-fat half-grid  minors}

Let $(\cV, \cE)$ be a model of the half-grid in a graph~$G$. Denote by $(\cV_K, \cE_K)$ the submodel of $(\cV, \cE)$ that consists precisely of all branch sets in $\cV$ which correspond to vertices $(n,m)$ of the half-grid with $n \geq K$ or $m \geq K$ and of precisely all branch paths in $\cE$ which correspond to edges $(n,m)(n',m')$ of the half-grid with $n,n' \geq K$ or $m,m'\geq K$, ie.\ $(\cV_K, \cE_K)$ corresponds to the induced subgraph of the half-grid with vertex set $\N^2\sm [K]^2$.
We say that $(\cV, \cE)$ is \defn{ultra-fat} if for every $K \in \N$ the model $(\cV_K, \cE_K)$ is $K$-fat. The half-grid is an \defn{ultra-fat minor} of~$G$, denoted by \defn{$HG \prec_{UF} G$}, if~$G$ contains an ultra-fat model of the half-grid.
For an end $\eps$ of~$G$, we write \defn{$HG \prec^\eps_{UF} G$} if $G$ contains an ultra-fat model $(\cV, \cE)$ of~$HG$ whose rays are all $\eps$-rays.

The following observation is immediate from the definitions:

\begin{obs} \label{obs:UFHGImpliesAsympHG}
    Let $\eps$ be an end of a graph $G$. If $HG \prec_{UF}^\eps G$, then $HG \prec_\infty^\eps G$. \qed
\end{obs}

Moreover, we have the following proposition:

\begin{prop} \label{prop:UFHGImpliesDivHG}
    Let $G$ be a graph, and let $(\cV, \cE)$ be an ultra-fat model of the half-grid in~$G$. If all its branch sets are finite, then $(\cV, \cE)$ diverges.
\end{prop}

\begin{proof}
    The proof is analogous to \cite[Proposition~3.2]{AH24}, we repeat the details here for clarity.

    Let $(x_n)_{n\in \N}$ and $(y_n)_{n \in \N}$ with $x_n, y_n \in V(HG) \cup E(HG)$ for all $n \in \N$ such that $d_{HG}(x_n, y_n) \rightarrow \infty$. Let $U_n, W_n$ be the branch sets or paths of $(\cV, \cE)$ corresponding to $x_n, y_n$, respectively. 
    Now suppose for a contradiction that $d_G(U_n, W_n)$ does not tend to infinity. Then we may assume, by restricting to subsequences, that $d_G(U_n, W_n) = K$ for some $K \in \N$ and all $n \in \N$. 
    Since $(\cV, \cE)$ is ultra-fat, this implies that, for all $n \in \N$, at least one of $x_n, y_n$ is contained in $HG[[K]^2]$.
    Hence, one of $(x_n)_{n \in \N}$ and $(y_n)_{n \in \N}$, say $(x_n)_{n \in \N}$, is eventually constant, so we may assume that $x_n = z$ for all $n \in \N$ and some $z \in V(HG) \cup E(HG)$. 
    Since $d_G(x_n, y_n) \rightarrow \infty$, it follows (after possibly restricting to a subsequence of~$(y_n)_{n \in \N}$) that the $y_n$ are pairwise distinct.
    Set $U_z := V_z$ if $z \in V(HG)$ or $U_z := V(E_z)$ if $z \in E(HG)$.
    As~$U_z$ is finite and $d_G(U_z, W_n) = K$ for all $n \in \N$, there is some $v \in U_z$ and an infinite index set $I \subseteq \N$ such that $d_G(v, W_n) = K$ for all $n \in I$. Hence, $d_G(W_n, W_m) \leq 2K$ for all $n,m \in \N$. But since $I$ is infinite and the $W_n$ are pairwise distinct, this contradicts that $(\cV, \cE)$ is ultra-fat.
\end{proof}

\section{Proof of Coarse Halin Grid Theorem}\label{sec:Halin:Grid}

In this section we prove two coarse versions of \ref{HalinsGridTheorem}. Let us first state them for locally finite graphs.

\begin{thm} \label{thm:Halin:HalfGrid}
    Let $\eps$ be an end of a \lf\ graph~$G$. 
    \begin{enumerate}[label=\rm{(\arabic*)}]
        \item \label{itm:Halin:HG:Asymp:Copy} If $\omega \cdot R \prec^\eps_{2K-1} G$ for some $K \in \N$, then $HG \prec^\eps_K G$. In particular, if $\omega \cdot R \prec^\eps_\infty G$, then $HG \prec^\eps_\infty G$.
        \item \label{itm:Halin:HG:UF:Copy} If $G$ contains $\eps$-rays $R_0, R_1, \dots$ such that $d_G(R_i, R_j) \geq \max\{i,j\}$ for all $i \neq j$, then $HG \prec_{UF}^\eps G$.
    \end{enumerate}
\end{thm}

\noindent We remark that we show in \cref{ex:Halin:HalfGrid:sharp} that the $2K-1$ in \ref{itm:Halin:HG:Asymp:Copy} is necessary.

The first part of \cref{thm:Halin:HalfGrid} immediately implies \cref{main:Halin:HalfGrid}, which we restate here for convenience.

\begin{customthm}{\cref*{main:Halin:HalfGrid}}
    \emph{Let $G$ be a one-ended, locally finite graph. If $\omega \cdot R$ is an asymptotic minor of $G$, then the half-grid is an asymptotic minor of $G$.}
\end{customthm}

\begin{proof}[Proof of \cref{main:Halin:HalfGrid} given \cref{thm:Halin:HalfGrid}]
    Since $G$ is one-ended, $\omega\cdot R \prec_\infty G$ implies $\omega\cdot R \prec_\infty^\eps G$ for the unique end $\eps$ of $G$. The assertion now follows from \cref{thm:Halin:HalfGrid}~\ref{itm:Halin:HG:Asymp:Copy}.
\end{proof}

In fact, we prove a more general version of \cref{thm:Halin:HalfGrid} (see \cref{thm:Halin:HalfGrid:NotLocFin,thm:Halin:HalfGrid:UltraFat} below) for not necessarily locally finite graphs, where we impose a further condition on the rays in the fat model of $\omega \cdot G$. 

Since a $K$-fat model of the half-grid contains between any two of its rays infinitely many paths that are pairwise at least $K$ apart, a necessary condition for containing a $K$-fat model of the half-grid whose rays are all $\eps$-rays is that there are infinitely many $\eps$-rays which pairwise cannot be separated by removing finitely many sets of vertices that all have diameter at most $K-1$. 
We will present in \cref{ex:Halin:HalfGrid:NotLocFin} a graph that verifies this statement, showing that \cref{thm:Halin:HalfGrid} is false for arbitrary (not locally finite) graphs without this additional assumption.
Conversely, we will prove that a slightly stronger condition already suffices to guarantee a $K$-fat half-grid minor.

Before we can state our two coarse versions of \ref{HalinsGridTheorem} for arbitrary graphs, we need another definition. We say that a ray $R = r_0r_1 \dots$ in a graph $G$ is \defn{$K$-fat} for some $K\in\N$ if for every $n \in \N$ there exists some $N \in \N$ such that $d_G(Rr_n, r_NR) \geq K$. It is easy to see that this is the case if and only if there exist (finite) sets $B_i \subseteq V(R)$ such that $\bigcup_{i \in \N} B_i = V(R)$ and such that the~$B_i$ with even~$i$ form the branch sets and the~$B_i$ with odd~$i$ form the branch paths of a $K$-fat model of the ray (cf.\ \cite[Lemma~6.2]{GP2023+}). Moreover, every ray in a locally finite graph is $K$-fat for every $K \in \N$.

\begin{thm} \label{thm:Halin:HalfGrid:NotLocFin}
    Let $\eps$ be an end of a graph~$G$, and let $K \in \N$. Suppose that there is an infinite collection~$\cR$ of $(2K-1)$-fat $\eps$-rays in~$G$ that are pairwise at least $2K-1$ apart and such that no two rays in~$\cR$ can be separated in~$G$ by removing finitely many balls of radius at most $4K-4$ in~$G$. Then $HG \prec^\eps_K G$.
\end{thm}

\begin{thm} \label{thm:Halin:HalfGrid:UltraFat}
    Let $\eps$ be an end of a graph~$G$, and suppose that there are $\eps$-rays $R_0, R_1, \dots$ in~$G$ that are $K$-fat for all $K\in\N$ and such that $d_G(R_i, R_j) \geq \max\{i,j\}$ for all $i \neq j \in \N$ and such that no two rays $R_i, R_j$ can be separated in~$G$ by removing finitely many balls of finite radius. Then $HG \prec^\eps_{UF} G$.
\end{thm}

Let us first show that \cref{thm:Halin:HalfGrid:NotLocFin,thm:Halin:HalfGrid:UltraFat} imply \cref{thm:Halin:HalfGrid}. 

\begin{proof}[Proof of \cref{thm:Halin:HalfGrid} given \cref{thm:Halin:HalfGrid:NotLocFin,thm:Halin:HalfGrid:UltraFat}]
    \ref{itm:Halin:HG:Asymp:Copy}: Take any model $(\cV, \cE)$ of $\omega \cdot R$ that witnesses ${\omega \cdot R \prec^\eps_{2K-1} G}$. Since every vertex of $\omega \cdot R$ has degree at most $2$, we may assume that the branch sets in~$\cV$ are finite paths which start and end in the endvertices of the incident branch paths. Then $(\cV, \cE)$ yields an infinite collection $\cR$ of $(2K-1)$-fat $\eps$-rays that are pairwise at least $2K-1$ apart. Since all rays in $\cR$ are $\eps$-rays, no two of them can be separated by removing finitely many vertices. Hence, as $G$ is locally finite, they can neither be separated by removing finitely many balls of radius at most $4K-4$. Thus, $\cR$ satisfies the premise of \cref{thm:Halin:HalfGrid:NotLocFin}, and so $HG \prec^\eps_K G$ as desired.

    \ref{itm:Halin:HG:UF:Copy}: Since $G$ is locally finite, every ray $R_i$ is $K$-fat for every $K \in \N$. Moreover, since all rays $R_i$ are $\eps$-rays, no two of them can be separated by removing finitely many vertices. As $G$ is locally finite, they can neither be separated by removing finitely many balls of finite radius. Thus, the $R_i$ satisfy the premise of \cref{thm:Halin:HalfGrid:UltraFat}, and so $HG \prec_{UF}^\eps G$ as desired.
\end{proof}

The remainder of this section is devoted to the proofs of \cref{thm:Halin:HalfGrid:NotLocFin,thm:Halin:HalfGrid:UltraFat}.
The \defn{hexagonal half-grid} is obtained from the half-grid by deleting every other rung, as shown in \cref{fig:HexGrid}.
We call the rays $S_i$ of the (hexagonal) half-grid its \defn{vertical rays} and the edges~$e_{i}$ its \defn{horizontal edges} (see \cref{fig:HexGrid}).

A \defn{subdivision} of a graph~$H$ is a graph which arises from~$H$ by replacing every edge in~$H$ by a new path between its endvertices such that no new path has an inner vertex in $V(H)$ or on any other new path. The original vertices of $H$ are the \defn{branch vertices} of the subdivision and the new paths are its \defn{branch paths}.
If $H$ is a subdivision of the (hexagonal) half-grid, we call its branch paths corresponding to horizontal edges \defn{horizontal paths}.

\begin{figure}[ht]
        \begin{tikzpicture}
    \tikzset{edge/.style = {->,> = latex'}}
    
%vertical rays
    \foreach \x in {-4,-2,0} {
    \draw[edge] (\x,0) to (\x,4.4);
    
     \draw[edge] (\x +1,0) to (\x +1 ,4.4);
     
     \foreach \y in {0.5,1.5,2.5,3.5} {
      \draw[fill,black] (\x,\y) circle (.05);
       \draw[fill,black] (\x+1,\y) circle (.05);
         \draw (\x,\y) to (\x+1,\y);
     }
    
        }

%other rungs
     \foreach \x in {-3,  -1 } {
         \foreach \y in {1,2,3,4}{

         \draw[fill,black] (\x,\y) circle (.05);
         
         \draw[fill,black] (\x +1 ,\y) circle (.05);
         
         \draw (\x,\y) to (\x+1,\y);
        }
         }
         
%ray 7

\draw[edge] (2,0) to (2,4.4);
   \foreach \y in {1,2,3,4}{

         \draw[fill,black] (1,\y) circle (.05);
          \draw[fill,black] (2,\y) circle (.05);
           \draw (1,\y) to (2,\y);
            \draw[fill,black] (2,\y-0.5) circle (.05);
             \draw (2,\y-0.5) to (2.5,\y-0.5);
         }

    \node at (3,2) {$\dots$};
     %\node at (3,1) {$\dots$};

%bottom vertices
    \foreach \x in {-4,...,2}{
     \draw[fill,black] (\x,0) circle (.05);
     }
     
%ray labels
\foreach \x in {0,...,6}{
    \node at (\x-3.8,-0.3) {\footnotesize{$S_{\x}$}};
     }
%edge labels
    %first collum
    
    \node at (-3.5, 0.3 ) {\footnotesize{$e_0$}};
     \node at (-3.5, 1.3 ) {\footnotesize{$e_3$}};
      \node at (-3.5, 2.3 ) {\footnotesize{$e_8$}};
       \node at (-3.5, 3.3 ) {\footnotesize{$e_{15}$}};
       
%second collum
    
    \node at (-2.5, 0.8 ) {\footnotesize{$e_2$}};
     \node at (-2.5, 1.8 ) {\footnotesize{$e_7$}};
      \node at (-2.5, 2.8 ) {\footnotesize{$e_{14}$}};
      
%third collum
    
    \node at (-1.5, 0.3 ) {\footnotesize{$e_1$}};
     \node at (-1.5, 1.3 ) {\footnotesize{$e_6$}};
      \node at (-1.5, 2.3 ) {\footnotesize{$e_{13}$}};
%fourth collum
    
    \node at (-.5, 0.8 ) {\footnotesize{$e_5$}};
     \node at (-.5, 1.8 ) {\footnotesize{$e_{12}$}};

%fith collum
    
    \node at (.5, 0.3 ) {\footnotesize{$e_4$}};
     \node at (.5, 1.3 ) {\footnotesize{$e_{11}$}};
     
\node at (1.5, 0.8 ) {\footnotesize{$e_{10}$}};
     \node at (2.5, .3 ) {\footnotesize{$e_{9}$}};

% vertices on R_1 ray 

\foreach \y in {1,2,3,4}{

         \draw[fill,black] (-4,\y) circle (.05);
         }
\end{tikzpicture}
    \vspace{-3em}
\caption{The hexagonal half-grid with vertical rays $S_i$ and horizontal edges~$e_{i}$.}
\label{fig:HexGrid}
\end{figure}

For the proofs of \cref{thm:Halin:HalfGrid:NotLocFin,thm:Halin:HalfGrid:UltraFat}, it will be more convenient to first construct a subdivision of the hexagonal half-grid and then contract it into a model of the half-grid.
Since \cref{thm:Halin:HalfGrid:NotLocFin,thm:Halin:HalfGrid:UltraFat} yield fat models of the half-grid, we want to find a subdivision of the hexagonal half-grid whose vertical rays~$S_i$ are fat and pairwise far apart and whose horizontal paths are far apart from all other horizontal paths and from all vertical rays except those two rays that contain its endvertices. 

For this, consider the following strengthening of \ref{HalinsGridTheorem} by Kurkofka, Melcher and Pitz.

\begin{thm}[\rm{\cite[Theorem~1]{KMP2022}}] \label{thm:KMP:HalinGrid}
    For every infinite collection $\cR$ of disjoint equivalent rays in a graph $G$ there is a subdivision of the hexagonal half-grid in~$G$ such that all its vertical rays belong to~$\cR$.
\end{thm}

\cref{thm:KMP:HalinGrid} comes somewhat close to what we want: If we apply \cref{thm:KMP:HalinGrid} to the collection $\cR$ of $\eps$-rays provided by the premise of \cref{thm:Halin:HalfGrid:NotLocFin} or \ref{thm:Halin:HalfGrid:UltraFat}, then we obtain a subdivision~$H$ of the hexagonal half-grid in~$G$ whose vertical rays belong to $\cR$, and which are thus fat and pairwise far apart.

However, we also need that the horizontal paths~$P_i$ of~$H$ are pairwise far apart and that they are far apart from all vertical rays~$S_i$ (except those two containing their endvertices). While the former property is easily achievable by restricting to a subset of the~$P_i$ at least for locally finite graphs, the latter property cannot be achieved by post-processing the given subdivision~$H$. Indeed, it might even happen that for all $i < j \in \N$ all $S_i$--$S_{i+1}$ paths $P_k$ in~$H$ have distance~$1$ from $S_j$, so we cannot obtain a hexagonal half-grid with the desired distance constraints by removing some of the paths and rays from~$H$. 

To solve this problem, we apply \cref{thm:KMP:HalinGrid} to a suitable auxiliary graph $G'$. Let us briefly sketch the idea for locally finite graphs. In that case, we may take as $G'$ the graph obtained from~$G$ by `contracting' large-radius balls around the rays $R \in \cR$ onto the rays~$R$. We then apply \cref{thm:KMP:HalinGrid} to~$\cR$ in~$G'$, which yields a subdivision~$H'$ of the hexagonal half-grid in~$G'$. By the definition of~$G'$, each horizontal path of~$H'$ `lifts' to a path in~$G$ that is far apart from all vertical rays of~$H'$ except for the two which it links. To obtain the desired subdivision~$H$ of the hexagonal half-grid in~$G$, we then only need to choose the horizontal paths for~$H$ among the `lifts' of the horizontal paths of~$H'$ so that they are pairwise far apart, which we can easily do recursively in locally finite graphs.
To obtain \cref{thm:Halin:HalfGrid:NotLocFin,thm:Halin:HalfGrid:UltraFat} also for graphs that are not locally finite, we still follow this approach, but we need to define $G'$ more carefully.

We will find the subdivisions of the hexagonal half-grid for \cref{thm:Halin:HalfGrid:NotLocFin,thm:Halin:HalfGrid:UltraFat} simultaneously, by showing the following lemma (\cref{lem:Halin:HG:Technical}). For the proofs of \cref{thm:Halin:HalfGrid:NotLocFin,thm:Halin:HalfGrid:UltraFat}, it then suffices to apply \cref{lem:Halin:HG:Technical} with $K_i := K$ for all $i \in \N$ or with $K_i := i$ for all $i \in \N$, and then contract the given subdivision of the hexagonal half-grid appropriately into a half-grid.

\begin{lem} \label{lem:Halin:HG:Technical}
    Let $K_0 \leq K_1 \leq \ldots \in \N$, and let $R_0, R_1, \dots$ be rays in a graph $G$ that are $(2K_i-1)$-fat for all $i\in\N$ such that $d_G(R_i, R_j) \geq K_i + K_j-1$ for all $i \neq j \in \N$ and such that no two rays $R_i, R_j$ can be separated in $G$ by deleting finitely many balls of radius $4K_k-4$ for any $k \in \N$. Then there exists a subdivision~$H$ of the hexagonal half-grid in~$G$ such that
    \begin{enumerate}[label=\rm{(\roman*)}]
        \item \label{itm:TechnicalHalinLemma:1} the set $\cS = \{S_0, S_1, \dots\}$ of vertical rays in $H$ is a subset of $\{R_0, R_1, \dots\}$,
        \item \label{itm:TechnicalHalinLemma:2} every horizontal path $P$ in $H$ with endvertices in $S_i, S_{i+1}$ has distance at least $K_k$ from $S_k$ in $G$ for all rays $S_k \in \cS\sm \{S_i, S_{i+1}\}$,
        \item \label{itm:TechnicalHalinLemma:3} for every horizontal path $P$ in $H$ with an endvertex in $S_i$, the subgraph\footnote{Recall that $G[X,r] := G[B_G(X,r)]$ for $X \subseteq G$ and $r \in \N$ (see \cref{sec:Prelim}).} $P \cap G[S_i, K_i-1]$ is a path of length $K_i-1$, and
        \item \label{itm:TechnicalHalinLemma:4} every two distinct horizontal paths in $H$ with endvertices in $S_i, S_{i+1}$ and $S_j, S_{j+1}$, respectively, are at least $K_n$ apart in $G$ where $n := \min\{K_{i},K_{j}\}$.
    \end{enumerate}
\end{lem}

\begin{proof}
    We first construct an auxiliary graph $G'$ to which we will apply \cref{thm:KMP:HalinGrid}. 
    For this, we recursively choose $R_{k_i}$--$R_{\ell_i}$ paths $P_i$ in~$G$, where $((k_i, \ell_i))_{i \in \N}$ is some sequence in~$\N^2$ containing every $(n,m) \in \N^2$ with $n \neq m$ infinitely often,
    such that for all $i < j \in \N$
    \begin{enumerate}[label=\rm{(\arabic*)}]
        \item \label{itm:HexGrid:Proof:NEW:Pi:1} $d_G(P_i, P_j) \geq 4K_j-3$,
        \item \label{itm:HexGrid:Proof:NEW:Pi:3} 
        if $d_G(P_i, r), d_G(P_j, s) < K_k$ for some $k \in \N$ and $r, s \in V(R_k)$, then $s$ appears on $R_k$ after $r$.
    \end{enumerate}
    Let $j \in \N$ and assume that we have already chosen $P_i$ for all $i < j$. 
    Since $d_G(R_k, R_\ell) \geq K_k + K_\ell-1$ for all $k\neq \ell \in \N$, there exists some $M \in \N$ such that no ray~$R_k$ with $k > M$ has distance less than~$K_k$ to any~$P_i$ with $i < j$.
    As the rays $R_k$ are $(2K_k-1)$-fat, there is for every $k \leq M$ some vertex~$r_k$ such that $r_kR_k$ has distance at least $K_k$ from all $P_i$ with $i<j$.
    Since deleting finitely many balls of radius $4K_n-4$ does not separate $R_{k_j}$ and $R_{\ell_j}$ in $G$ for any $n \in \N$ by assumption, there is an $R_{k_j}$--$R_{\ell_j}$ path~$P_j$ in $G$ which avoids $\bigcup_{i<j}B_G(P_i, 4K_j-4)$ and $\bigcup_{k \leq M} B_G(R_kr_k, K_k-1)$. Clearly, $P_j$ satisfies \ref{itm:HexGrid:Proof:NEW:Pi:1} and \ref{itm:HexGrid:Proof:NEW:Pi:3}.

    We now define an auxiliary graph $G'$. Its vertex set is $V(G') := \bigcup_{i \in \N} (V(R_i) \cup V(P'_i))$ where $P'_i := P_i-\bigcup_{k \in \N} B_G(R_k, K_k-1)$ for all $i \in \N$. We put an edge $uv$ between $u, v \in V(G')$ in the following three cases: 1) if $uv \in E(G)$, 2) if $u \in V(R_k)$ and $v \in V(P'_i)$ for some $k, i \in \N$ and $d_G(u,v) = K_k$, and 3) if $u \in V(R_k)$ and $v \in V(R_\ell)$ for some $k \neq \ell \in \N$ and some $B_G(u, K_k-1)$--$B_G(v, K_\ell-1)$ edge is contained in $P_i$ for some $i \in \N$. 

    By definition, $R_i \subseteq G'$ for all $i \in \N$. We claim that the $R_i$ are still equivalent in $G'$. 
    Indeed, let $R_k, R_\ell$ be given, and let $P_i, P_j$, with $i < j$, be $R_k$--$R_{\ell}$ paths. Let $I := \{n \in \N : d_G(R_n, P_i) < K_n\}$. 
    Since $P_i$ is connected, it follows by the definition of~$G'$ that $G'_i := G'[V(P'_i) \cup \bigcup_{n \in I} V(R_n)]$ is connected. Let~$Q_i$ be an $R_{k}$--$R_{\ell}$ path in~$G'_i$, and let~$Q_j$ be defined analogously for~$P_j$. Then $Q_i, Q_j$ are disjoint. Indeed, $P'_i$ and~$P'_j$ are disjoint by \ref{itm:HexGrid:Proof:NEW:Pi:1}, and any subpath of $Q_i$ in some~$R_n$ is disjoint from any subpath of $Q_j$ in~$R_n$ since the endvertices of edges of $G'_j$ incident with $R_n$ appear on $R_n$ after the endvertices of edges of $G'_i$ incident with $R_n$ by \ref{itm:HexGrid:Proof:NEW:Pi:3}. Since infinitely many $P_i$ are $R_k$--$R_\ell$ paths, $R_k, R_\ell$ are equivalent in~$G'$.

    We now apply \cref{thm:KMP:HalinGrid} in $G'$ to $\cR := \{R_0, R_1, \dots\}$, which yields a subdivision~$H'$ of the hexagonal half-grid in~$G'$. By \cref{thm:KMP:HalinGrid}, the vertical rays of~$H'$ are a subset $\cR'$ of $\cR$.
    We first choose rays $R_{i_0}, R_{i_1}, \ldots \in \cR'$ such that $i_0 < i_1 < \dots$ and such that, for all $j < k \in \N$, the ray $R_{i_k}$ appears in $H'$ as a vertical ray after $R_{i_j}$. Set $S_j := R_{i_j}$ for all $j \in \N$, and let $\cS := \{S_0, S_1, \dots\}$. Further, set $K'_j := K_{i_j}$ for all $j \in \N$, and note that $K'_j \geq K_j$. It is now straightforward to find a subdivision $H'' \subseteq H' (\subseteq G')$ of the hexagonal half-grid whose $i$-th vertical ray is $S_i$ for every $i \in \N$.
    
    We now turn~$H''$ into a subdivision~$H$ of the hexagonal half-grid in $G$. Its vertical rays will be precisely the rays in $\cS$ (in that order), and hence $H$ will satisfy~\ref{itm:TechnicalHalinLemma:1}.
    Before we choose suitable horizontal paths for~$H$, we first turn every horizontal path in~$H''$ into a path in~$G$. For this, let $W'_i = w'_0 \dots w'_n$ be a horizontal path of~$H''$ with endvertices in $S_{k_i}, S_{k_i+1} \in \cS$. By the definition of~$G'$, every edge~$f$ of~$W'_i$ which is not an edge of~$G$ is either of the form~$uv$ for some vertex~$v$ in some $R_j \in \cR$ and a vertex $u \in B_G(v, K_j)\sm B_G(v, K_j-1)$ or of the form $uv$ for vertices $u,v$ on distinct rays $R_j, R_k \in \cR$ such that $B_G(u,K_j-1)$ and $B_G(v,K_k-1)$ are joined by an edge~$e$ in some~$P_\ell$. In the former case we replace~$uv$ by a $u$--$v$ path in~$G$ of length~$K_j$, and in the latter case we replace~$uv$ by a $u$--$v$ path in~$G$ of length~$K_j + K_k-1$ containing $e$. We then let~$W_i$ be some $w'_0$--$w'_n$ path in the resulting $w'_0$--$w'_n$ walk in~$G$.

    Since $W'_i$ is a horizontal path of $H''$, it avoids all rays $S_m \in \cS$ except for $S_{k_i}, S_{k_i+1}$. By the definition of~$G'$, it follows that $V(W'_i)$ has distance at least~$K'_m$ in~$G$ from all $S_m \in \cS$ with $S_m \neq S_{k_i}, S_{k_i+1}$. 
    As $d_G(R_j, R_k) \geq K_j+K_k-1$ for all $R_j \neq R_k \in \cR$, and every vertex of $W_i-V(W'_i)$ has distance at most $K'_{k_i}-1$ or $K'_{k_i+1}-1$ from $S_{k_i}, S_{k_i+1}$, respectively, or distance at most $K_j-1$ from some ray $R_j \in \cR\sm \cS$, the path~$W_i$ in~$G$ has distance at least~$K'_j$ from $S_j$ for every $S_j \in \cS \sm \{S_{k_i}, S_{k_i+1}\}$.

    We now inductively define $H$ by choosing a branch path $E_i$ for every horizontal edge $e_i$ of the hexagonal half-grid, in the order $e_0, e_1, e_2, \ldots$ indicated in \cref{fig:HexGrid}. 
    In fact, we will choose~$E_i$ among the paths~$W_j$ joining the respective vertical rays of~$H$. By the construction above and because $K'_i \geq K_i$ for all $i \in \N$, this immediately implies that $H$ will satisfy \ref{itm:TechnicalHalinLemma:2} and \ref{itm:TechnicalHalinLemma:3}.

    So let $j \geq 0$ be given, and assume that for every $i < j$, we have already chosen some $W_{\ell_i}$ to be the branch path~$E_i$ corresponding to the edge~$e_i$ of the hexagonal half-grid. Let $k_{j} \in \N$ such that $S_{k_j}, S_{k_j+1}$ correspond to the vertical rays of the hexagonal half-grid that are joined by~$e_j$. 
    Since $d_G(P_i, P_k) \geq 4K_{k_j}-3$ for all $i,k \geq k_j$ by \ref{itm:HexGrid:Proof:NEW:Pi:1}, every vertex of some $E_i$ has distance less than $2K_{k_j}-1$ to at most one path $P_k$ with $k \geq k_j$. As $\bigcup_{i < j} V(E_i)$ is finite, it follows that at most finitely many~$P_k$ have distance less than $2K_{k_j}-1$ from some $E_i$ with $i<j$.
    Since the paths $W'_\ell$ are disjoint, all but finitely many of them avoid these finitely many $P_k$. Hence, by the definition of the $W_\ell$, all but finitely many of the paths $W_\ell$ that link $S_{k_j}$ and $S_{k_j+1}$ in~$H''$ satisfy the following condition:
    \begin{enumerate}[label=\rm{(\alph*)}]
        \item \label{itm:HexGrid:Proof:NEW:a} For every $k \in \N$, if $P_k \cap W_\ell \neq \emptyset$, then $d_G(P_k, E_i) \geq 2K_{k_j}-1$ for all $i < j$.
    \end{enumerate}
    
    As $d_G(R_i, R_k) \geq K_i+K_k-1$ for all $i \neq k \in \N$, and because the $E_i$ are finite, there is some $L \in \N$ such that $d_G(R_k, E_i) \geq K_{k_j}$ for all $k > L$. Since the $R_k$ are $(2K_{k_j}-1)$-fat, at most finitely many vertices of~$R_k$ have distance less than $K_{k_j}$ from some~$E_i$. Hence, and again because the $R_k$ are $(2K_k-1)$-fat, there exists for every $k \in \N$ some $r_k \in V(R_k)$ such that, for every $i < j$, the tail $r_kR_k$ has distance at least $K_{k_j}$ from $E_i$, and distance at least $2K_k-1$ from all vertices of $R_k$ that have distance less than $K_k$ from $E_i$. Then $\bigcup_{k<L} R_kr_k$ is finite, and hence, as the paths $W'_\ell$ are disjoint, all but finitely many of them avoid $\bigcup_{k<L} R_kr_k$. Thus, by the definition of the $W_\ell$, all but finitely many of the paths $W_\ell$ that link $S_{k_j}$ and $S_{k_j+1}$ in $H''$ satisfy the following conditions:
    \begin{enumerate}[label=\rm{(\alph*)}]
        \setcounter{enumi}{1}
        \item \label{itm:HexGrid:Proof:NEW:b} For every $k \in \N$, if $R_k \cap W_\ell \neq \emptyset$, then $d_G(R_k \cap W_\ell, E_i) \geq K_{k_j}$ for all $i < j$.
        \item \label{itm:HexGrid:Proof:NEW:c} 
        For every $i,k \in \N$ with $i < j$, and every $r, s \in V(R_k)$, if $B_G(r, K_k-1) \cap V(W_\ell) \neq \emptyset$ and $B_G(s, K_k-1) \cap V(E_i) \neq \emptyset$, then $d_G(r,s) \geq 2K_k-1$.
    \end{enumerate}

    Since, for all three conditions \ref{itm:HexGrid:Proof:NEW:a} to \ref{itm:HexGrid:Proof:NEW:c} separately, all but finitely many of the paths~$W_\ell$ that link~$S_{k_j}$ and~$S_{k_j+1}$ in~$H''$ satisfy the respective condition, there is some such~$W_\ell$ that satisfies \ref{itm:HexGrid:Proof:NEW:a} to \ref{itm:HexGrid:Proof:NEW:c} simultaneously, and whose endvertices in $S_{k_j}, S_{k_j+1}$ appear on $S_{k_j}, S_{k_j+1}$, respectively, after all endvertices of branch paths~$E_i$ with $i < j$. We set $\ell_j := \ell$ and $E_j := W_\ell$.

    In this way, we construct a subdivision $H$ of the hexagonal half-grid in $G$ that satisfies \ref{itm:TechnicalHalinLemma:1} to~\ref{itm:TechnicalHalinLemma:3}. We claim that~$H$ also satisfies \ref{itm:TechnicalHalinLemma:4}. Indeed, let $i < j$ be given, and let $v_j \in V(E_j)$. If $v_j \in B_G(P_k, K_{k_j}-1)$ for some $k \in \N$, then $d_G(v_j, E_i) \geq K_{k_j}$ follows by~\ref{itm:HexGrid:Proof:NEW:a}. Hence, as $E_j = W_{\ell_j}$ and by the definition of~$W_{\ell_j}$, every other vertex $v_j \in V(E_j)$ is contained in $B_G(R_k, K_k-K_{k_j})$ for some $k \in \N$. 
    If $v_j \in V(R_k)$, then $d_G(v_j, E_i) \geq K_{k_j}$ follows by~\ref{itm:HexGrid:Proof:NEW:b}. Hence, we may assume $v_j \in B_G(R_k, K_k-K_{k_j})\sm V(R_k)$, which in particular implies that $K_{k_j} < K_k$.
    Let also $v_i \in V(E_i)$. If $v_i \notin B_G(R_k, K_k-1)$, then $d_G(v_i, v_j) \geq K_{k_j}$ follows since $v_j \in B_G(R_k, K_k-K_{k_j})$ and $K_k > K_{k_j}$. Hence, we may assume that $v_i \in B_G(R_k, K_k-1)$. Let $r,s \in V(R_k)$ such that $v_i \in B_G(r, K_k-1)$ and $v_j \in B_G(s, K_k-K_{k_j})$. Then $d_G(r,s) \geq 2K_{k}-1$ by~\ref{itm:HexGrid:Proof:NEW:c}, and thus $d_G(v_i, v_j) \geq d_G(r,s)-d_G(v_i,r)-d_G(v_j,s) \geq (2K_k-1) - (K_k-1) - (K_k-K_{k_j}) = K_{k_j}$.
    This completes the proof that $H$ satisfies \ref{itm:TechnicalHalinLemma:4}, and hence concludes the proof.
\end{proof}

We can now prove \cref{thm:Halin:HalfGrid:NotLocFin,thm:Halin:HalfGrid:UltraFat}.

\begin{proof}[Proof of \cref{thm:Halin:HalfGrid:NotLocFin}]
    Without loss of generality assume that the collection $\cR$ is countably infinite and enumerated as $\cR := \{R_0, R_1, \dots\}$. By the assumptions on $\cR$, the rays $R_i$ satisfy the premise of \cref{lem:Halin:HG:Technical} with $K_i := K$ for all $i \in \N$. Hence, applying \cref{lem:Halin:HG:Technical} yields a subdivision $H$ of the hexagonal half-grid in $G$ that satisfies \ref{itm:TechnicalHalinLemma:1} to \ref{itm:TechnicalHalinLemma:4}.
    It is now straight forward to check that~$H$ can be contracted into a $K$-fat model $(\cV, \cE)$ of the half-grid (where the branch paths in $\cE$ corresponding to horizontal edges of the half-grid are a subset of the horizontal paths of $H$, and the branch sets and paths corresponding to vertices and edges of the $i$-th vertical ray $\{i\}\times \N$ of the half-grid are suitable subpaths of $S_i$). Indeed, \ref{itm:TechnicalHalinLemma:1} combined with $d_G(R_i, R_j) \geq 2K-1 \geq K$ for all $R_i \neq R_j \in \cR$ ensures that branch sets and paths of vertices and edges of the half-grid on distinct vertical rays have distance at least~$K$ in~$G$. Moreover, \ref{itm:TechnicalHalinLemma:4} ensures that every two distinct horizontal branch paths have distance at least~$K$, and \ref{itm:TechnicalHalinLemma:2} yields that horizontal branch paths have distance at least~$K$ from branch sets and paths on vertical rays that are not joined by that branch path. Finally, \ref{itm:TechnicalHalinLemma:1} together with the rays in $\cR$ being $(2K-1)$-fat ensures that we may choose $(\cV, \cE)$ so that branch sets and paths on the same ray $S_i$ have distance at least $K$ from each other, and by \ref{itm:TechnicalHalinLemma:3} we can also choose them so that they have distance at least $K$ from branch paths corresponding to non-incident horizontal edges that have an endvertex on the same vertical ray.
\end{proof}

\begin{proof}[Proof of \cref{thm:Halin:HalfGrid:UltraFat}]
    By restricting the sequence $(R_i)_{i \in \N}$ to those rays $R_i$ whose index is even and renumbering them (i.e.\ `$R_i := R_{2i}$'), we may assume that $d_G(R_i, R_j) \geq \max\{2i,2j\} \geq i+j-1$. Then the $R_i$ satisfy the premise of \cref{lem:Halin:HG:Technical} with $K_i := i$ for all $i \in \N$. Hence, applying \cref{lem:Halin:HG:Technical} yields a subdivision~$H$ in~$G$ of the hexagonal half-grid that satisfies \ref{itm:TechnicalHalinLemma:1} to \ref{itm:TechnicalHalinLemma:4}. 
    Analogously to the previous proof (of \cref{thm:Halin:HalfGrid:NotLocFin}), it is straightforward to check that $H$ can be contracted into a model $(\cV, \cE)$ of the half-grid where the fact that $K_i = i$ instead of $K_i = K$ and that the rays $R_i$ are $K_i$-fat for all $i\in\N$ yields that for every $i \in \N$, the submodel of $(\cV, \cE)$ corresponding to $\N_{\geq i} \times \N$ is $i$-fat. 

    \begin{figure}[ht]
        \scalebox{0.8}{%
\begin{tikzpicture}
    \tikzset{edge/.style = {->,> = latex'}}

%gray grid
    
%vertical rays
    \foreach \x in {0,1,2,3,4,5,6,7,8,9,10,11,12} {
    \draw[lightgray,edge] (\x,0) to (\x,7.75);}

%horizontal rays
    \foreach \y in {0,1,2,3,4,5,6,7} {
    \draw[lightgray] (0,\y) to (12,\y);}

%black grid

%vertical edges

    \draw[black,line width=1.3] (2,0) to (2,1);
    \foreach \y in {0,1,2} {
    \draw[black,line width=1.3] (6-\y,\y) to (6-\y,\y+1);}
    \foreach \y in {0,1,2,3,4} {
    \draw[black,line width=1.3] (10-\y,\y) to (10-\y,\y+1);}
    \foreach \y in {2,3,4,5,6} {
    \draw[black,line width=1.3] (14-\y,\y) to (14-\y,\y+1);}
    \draw[black,line width=1.3] (12,6) to (12,7);

%horizontal rays
    \foreach \y in {0,1,2,3,4,5,6,7} {
    \draw[black,edge,line width=1.3] (0+\y,\y) to (12.75,\y);}

%vertices

    \draw[fill,black,line width=1.3] (2,0) circle (.05);
    \draw[fill,black,line width=1.3] (2,1) circle (.05);
    \foreach \y in {0,1,2} {
    \draw[fill,black,line width=1.3] (6-\y,\y) circle (.05);
    \draw[fill,black,line width=1.3] (6-\y,\y+1) circle (.05);}
    \foreach \y in {0,1,2,3,4} {
    \draw[fill,black,line width=1.3] (10-\y,\y) circle (.05);
    \draw[fill,black,line width=1.3] (10-\y,\y+1) circle (.05);}
    \foreach \y in {2,3,4,5,6} {
    \draw[fill,black,line width=1.3] (14-\y,\y) circle (.05);
    \draw[fill,black,line width=1.3] (14-\y,\y+1) circle (.05);}
    \draw[fill,black,line width=1.3] (12,6) circle (.05);
    \draw[fill,black,line width=1.3] (12,7) circle (.05);
    \foreach \y in {0,1,2,3,4,5,6,7} {
    \draw[fill,black,line width=1.3] (\y,\y) circle (.05);}

%vertex labels

    \node at (0.45,-0.3) {\small{$(0,0)$}};
    \node at (1.45,1.3) {\small{$(1,0)$}};
    \node at (2.45,2.3) {\small{$(2,0)$}};
    \node at (3.45,3.3) {\small{$(3,0)$}};
    \node at (4.45,4.3) {\small{$(4,0)$}};
    %\node at (4.6,5.3) {\small{$(6,1)$}};
    \node at (6.45,6.3) {\small{$(6,0)$}};
    \node at (7.45,7.3) {\small{$(7,0)$}};

    \node at (2.45,-0.3) {\small{$(0,1)$}};
    \node at (2.45,1.3) {\small{$(1,1)$}};
    \node at (4.45,2.3) {\small{$(2,1)$}};
    \node at (4.45,3.3) {\small{$(3,1)$}};
    \node at (6.45,4.3) {\small{$(4,1)$}};
    \node at (6.45,5.3) {\small{$(5,1)$}};
    \node at (8.45,6.3) {\small{$(6,1)$}};
    \node at (8.45,7.3) {\small{$(7,1)$}};

    \node at (6.45,-0.3) {\small{$(0,2)$}};
    \node at (10.45,-0.3) {\small{$(0,3)$}};

%7-fat

    \draw[blue,line width=1.3] (12.75,8) [rounded corners=15pt] -- (5.5,8) -- (5.5,-0.6) -- (12.75,-0.6);
    \node at (7,8.25) {{\color{blue}$7$-fat}};

\end{tikzpicture}}
        \vspace{-3em}
        \caption{The model $(\cV', \cE')$ of the hexagonal half-grid in the proof of \cref{thm:Halin:HalfGrid:UltraFat}.}
        \label{fig:UltraFatHG}
    \end{figure}

    We now use $(\cV, \cE)$ to find an ultra-fat model of the half-grid. For this, we first define a model $(\cV', \cE')$ of the hexagonal half-grid as indicated in \cref{fig:UltraFatHG}. For every vertex $(n,m) \in \N^2$, its branch set $V'_{(n,m)}$ is equal to $V_{(j,n)}$ for some suitable $j \in \N$, every horizontal branch path $E'_{(n,m)(n+1,m)}$ is equal to some `vertical' branch path of the form $E_{(i,j)(i,j+1)}$, and every `vertical' branch path $E'_{(n,m)(n,m+1)}$ is contained in $\bigcup_{j\in\N} (G[V_{(j,n)}] \cup E_{(j,n)(j+1,n)})$ (see \cref{fig:UltraFatHG}).
    As the submodel of $(\cV, \cE)$ corresponding to $\N_{\geq K} \times \N$ is $K$-fat for all $K$, the submodel of $(\cV', \cE')$ corresponding to $\N^2\sm [K]^2$ is $K$-fat for all $K \in \N$. It is now easy to check that $(\cV', \cE')$ can be contracted into a model of the half-grid which still has this property, and which is hence ultra-fat.
\end{proof}

\section{Counterexamples regarding coarse versions of Halin's Grid Theorem}\label{sec:CounterEx}

In this section, we present two examples regarding coarse versions of Halin's Grid Theorem. 
The first example shows that \cref{prob:HalinGrid} is false in general for graphs that are not locally finite. 
The second example shows that the number $2K-1$ in \cref{thm:Halin:HalfGrid} is best possible.

\begin{ex}\label{ex:Halin:HalfGrid:NotLocFin}
    \emph{There is a one-ended graph $G$ such that $\omega \cdot R \prec_\infty G$ but $HG \not\prec_2 G$.}
\end{ex}

\begin{proof}
    Let $G$ be the graph obtained from the disjoint union of infinitely many rays $R_0, R_1, R_2, \ldots$ and a~$K_{\aleph_0}$ by joining the $i$-th vertices of the $R_j$ with the $i$-th vertex of $K_{\aleph_0}$ by paths $P_{i,j}$ of length $j$ for all $i,j \in \N$.
    Then for all $K \in \N$ the rays $R_K,R_{K+1},\ldots$ are $2K$-fat and pairwise $2K$ apart, and thus can be contracted into a $2K$-fat model of $\omega \cdot R$. In particular, as $G$ is one-ended, we have $\omega \cdot R \prec_\infty^\eps G$ for the unique end~$\eps$ of~$G$.
    
    We claim that the half-grid is not a $2$-fat minor of~$G$. Indeed, suppose for a contradiction that~$G$ contains a $2$-fat model $(\cV, \cE)$ of the half-grid. As the $K_{\aleph_0}$ has diameter $1$, at most one branch set in~$\cV$ and at most one branch path in~$\cE$ can contain a vertex from~$K_{\aleph_0}$. But since the half-grid still contains a half-grid after removing an arbitrary vertex and an arbitrary edge, and because removing the $K_{\aleph_0}$ from $G$ leaves only combs $C_j = R_j \cup \bigcup_{i \in \N} P_{i,j}$ as components, one of the combs $C_j$ would contain a $K$-fat model of the half-grid. But as combs are trees, they not even contain a $0$-fat model of the half-grid. This yields the desired contradiction.
\end{proof}

\begin{ex}\label{ex:Halin:HalfGrid:sharp}
\emph{For all $K \in \N$, there is a one-ended, locally finite graph~$G$ such that $\omega\cdot R \prec_{2K-2} G$ but $HG \not\prec_K G$.}
\end{ex}

\begin{proof}
    Let $G$ be the graph obtained from the disjoint union of infinitely many rays $T, R^0, R^1, R^2, \ldots$ by joining the vertex $r_i^j$ of~$R^j = r_0^jr_1^j\dots$ and the vertex $t_{i+j}$ of $T = t_0t_1\dots$ by a path $P_{ij}$ of length~$K-1$ for all $i,j \in \N$. Then $T$ and all rays $R^j$ are equivalent, and $G$ is one-ended.
    Note that $G$ is locally finite since $\deg_G(r_i^j) \leq 3$ for all $i, j \in \N$, since all internal vertices of the paths $P_{ij}$ have degree~$2$ in~$G$, and since
    \[
    \deg_G(t_i) \leq |\{(j,k) \in \N^2 \mid j+k = i\}|+2 \leq |\{(j,k) \in \N^2 \mid j,k\leq i\}|+2 = (i+1)^2+2
    \]
    for all $i \in \N$. 
    Since $d_G(R^i, R^j) = 2K-2$ by construction, and since the rays $R_i$ are $(2K-2)$-fat as $G$ is locally finite, the rays $R^i$ can be contracted into a $(2K-2)$-fat model of $\omega \cdot R$. 
    
    We claim that the half-grid is not a $K$-fat minor of~$G$.
    For this, suppose towards a contradiction that there is a $K$-fat model $(\cV, \cE)$ of the half-grid in~$G$. Let~$S'_i$ be the induced subgraph of~$G$ on those vertices that lie in some branch set or branch path in $\cV \cup \cE$ belonging to the vertical ray~$S_i$ of the half-grid. Let further $U_n := \{t_n\} \cup \{r_i^j : i,j \in \N, i+j=n\}$, for $n \in \N$, be the $n$-th `level' of $G$, and note that $G-U_n$ has precisely one infinite component $C_n$, and that $C_{n+1}  \subseteq C_n$ for all $n \in \N$. As the $S'_i$ are connected and $\bigcap_{n \in \N} C_n = \emptyset$, there exists for every $S'_i$ some $n \in \N$ such that $S'_i$ meets all `levels' $U_m$ with $m \geq n$.
    Hence, there is some $n \in \N$ such that at least three distinct $S'_i$ meet every $U_m$ for $m \geq n$.
    Then for every $m \geq n$ the vertex $t_m$ in~$T$ cannot be in any $S'_i$ since it has distance less than $K$ to all the other vertices in~$U_m$.
    But similarly,~$t_m$ cannot lie in a branch path in~$\cE$ representing a horizontal edge of the half-grid since~$t_m$ has distance less than~$K$ from at least three~$S'_i$.
    So $(\cV, \cE)$ avoids the ray $T_{\geq n}$, 
    and hence at most~$n$ branch sets and paths of $(\cV, \cE)$ contain a vertex of $T$. Since the half-grid still contains the half-grid as a minor after deleting finitely many vertices and edges, and because the half-grid is connected, there is some component $C$ of $G-T$ which contains the half-grid as a minor. But the components of $G-T$ are all combs of the form $R^i\cup \bigcup_{j \in \N} P_{ij}$, and hence not even contain a $0$-fat model of the half-grid. This yields the desired contradiction.
\end{proof}

\section{Half-grid minors in one-ended \qt\ graphs}\label{sec:QT}

In this section we apply one of our coarse versions of Halin's Grid Theorem to prove \cref{main:HG:QT:OneEnded}, which we restate here for convenience:

\begin{customthm}{\cref*{main:HG:QT:OneEnded}} \label{thm:QTOneEnd:HalfGrid}
    \emph{Every one-ended, \qt, \lf\ graph~$G$ contains the half-grid as an ultra-fat minor.
    In particular, the half-grid is an asymptotic minor and a diverging minor of~$G$.}
\end{customthm}

A graph $G$ is \defn{\qt} if the automorphism group of~$G$ acts on $V(G)$ with only finitely many orbits, that is, if $V(G)$ can be partitioned into finitely many sets $U_0, \dots, U_n$ such that for all $i \in \{0, \dots, n\}$ and $u, v \in U_i$ there exists an automorphism~$\phi$ of~$G$ such that $\phi(u) = v$.

Given an automorphism $\phi$ of a graph $G$, a set $U\subseteq V(G)$ is \defn{$\phi$-invariant} if $\phi(U) = U$. An induced subgraph~$X$ of~$G$ is \defn{$\phi$-invariant} if $V(X)$ is $\phi$-invariant. An automorphism $\phi$ of a graph~$G$ is \defn{elliptic} if~$G$ contains a finite $\phi$-invariant set of vertices and \defn{non-elliptic} otherwise.
If $R=\ldots r_{-1}r_0 r_1\ldots$ is a $\varphi$-invariant double ray and $\phi$ is non-elliptic, then we say that $\phi$ \defn{acts as a translation} on~$R$, which means that there exists some $k \in \N$ such that $\phi(r_i) = \phi(r_{i+k})$ for all $i \in \Z$.

The main effort of this section goes into showing that every graph $G$ as in \cref{thm:QTOneEnd:HalfGrid} contains infinitely many rays $R_0, R_1, \dots$ such that $d_G(R_i, R_j) \geq \max\{i,j\}$ for all $i \neq j \in \N$. \cref{thm:QTOneEnd:HalfGrid} then follows from \cref{thm:Halin:HalfGrid}~\ref{itm:Halin:HG:UF:Copy}.
For this, we distinguish two cases.
In the first case, in \cref{subsec:NonElliptic} below, we assume that some automorphism of~$G$ is non-elliptic. The second case, in \cref{subsec:Elliptic} below, deals with the remaining case that all automorphisms of~$G$ are elliptic.
The formal proof of \cref{thm:QTOneEnd:HalfGrid}, which collects the results from \cref{subsec:Elliptic,subsec:NonElliptic}, can be found in \cref{subsec:ProofOfThm1}

\subsection{Some automorphism of \texorpdfstring{$G$}{G} is non-elliptic} \label{subsec:NonElliptic}

We now show that every one-ended, connected, \lf\ graph~$G$ with a non-elliptic automorphism contains infinitely many rays that are pairwise far apart.
For this, we need the following result of Halin:

\begin{lem}[\rm{\cite[Theorem~7]{H73}}] \label{lem:NonElliptic:ExistDoubleRay}
    For every non-elliptic automorphism $\phi$ of a connected graph~$G$ there exists a $\phi$-invariant double ray.
\end{lem}

Let $G$ and~$H$ be graphs.
A map $f\colon V(H) \to V(G)$ is a \defn{coarse embedding}, if there exist functions $\rho^-\colon [0, \infty) \to [0, \infty)$ and $\rho^+\colon [0, \infty) \to [0, \infty)$ such that $\rho^-(a) \to \infty$ for $a \to\infty$ and
\[
\rho^-(d_H(u, v)) \leq d_G(f(u), f(v)) \leq \rho^+(d_H(u, v))
\]
for all $u, v \in V(H)$. If $H$ is a subgraph of~$G$, then~$H$ is \defn{coarsely embedded} in~$G$ if the inclusion map $\iota\colon V(H) \rightarrow V(G)$ is a coarse embedding.

The next lemma guarantees that the double ray from \cref{lem:NonElliptic:ExistDoubleRay} is coarsely embedded in~$G$.

\begin{lem} \label{lem:NonElliptic:DivergingDoubleRay}
    Let $G$ be a \lf\ graph. Let $\phi$ be a non-elliptic automorphism of~$G$, and let $R$ be a $\phi$-invariant double ray.
    Then $R$ is coarsely embedded in $G$. 
\end{lem}

\begin{proof}
    For $L\geq 0$, we set $K_L:=\min\{d_G(x,y)\mid x,y \in V(R), d_R(x,y)=L\}$, and $\rho\colon\N\to\N, L\mapsto K_L$. Then
    \[
    \rho(d_R(x,y))=K_{d_R(x,y)}\leq d_G(x,y)\leq d_R(x,y).
    \]
    Hence, it remains to show that $K_L$ tends to infinity for $L \rightarrow \infty$. For this, let $n \in \N$ be given, and let $x_1,\ldots,x_m$ be representatives of the $\langle\phi\rangle$-orbits on~$R$. Since $G$ is locally finite, the balls $B_G(x_i, n)$ are finite. As there are only finitely many $\langle\phi\rangle$-orbits on~$R$, this implies that there exists some $N_n$ such that $d_R(x_i, r) \leq N_n$ for all $i \leq m$ and $r \in B_G(x_i, n) \cap V(R)$. In particular, as $x_1,\ldots,x_m$ are representatives of the $\langle\phi\rangle$-orbits on~$R$, we have $d_R(r,s) \leq N_n$ for all $r,s \in V(R)$ with $d_G(r,s) \leq n$. This implies that $K_{N_n+1} > n$, and thus $K_L$ tends to infinity for $L \rightarrow \infty$, as desired.
\end{proof}

Combining \cref{lem:NonElliptic:DivergingDoubleRay,lem:NonElliptic:ExistDoubleRay} we obtain that every \lf, connected graph with a non-elliptic automorphism contains a $\phi$-invariant double ray~$R$, and $R$ is coarsely embedded in~$G$. The next two lemmas show that if $G$ is one-ended, and hence all tails of $R$ lie in the same end, then we can find infinitely many rays $R_i$ in~$G$ in `thickened cylinders' around~$R$ of the form $G[R, L_{i+1}]-B_G(R, L_i+i+1)$ for suitable integers $L_0 < L_1 < \ldots \in \N$ with $L_{i+1} - L_{i} \geq i +1$. Hence, $d_G(R_i, R_j) \geq |L_j-L_i|  \geq \max\{i,j\}$ for all $i \neq j \in \N$ as desired. 

Let $G$ be a graph, and let $R$ be a double ray in~$G$.
For $K\in\N$, let $C$ be a component of~$G-B_G(R,K)$.
We say that $C$ is \defn{half-long} if there are infinitely many $r \in V(R)$ such that $C$ has a neighbour in $B_G(r,K)$.
Further, $C$ is \defn{half-thick} if, for
every $L \geq K$, some half-long component of $G - B_G(R, L)$ is contained in~$C$.

\begin{lem}\label{lem:withNonElliptic:halfthick}
    Let $G$ be a one-ended, \lf\ graph. Let $\phi$ be a non-elliptic automorphism of~$G$, and let $R$ be a $\phi$-invariant double ray.
    Then $G-B_G(R,L)$ has at least one half-thick component for every $L \in \N$.
\end{lem}

\begin{proof}
    Set $R =: \ldots r_{-1}r_0r_1 \ldots$, and let $L \in \N$.
    Let us first show that $G[R, K] \neq G$ for all $K \in \N$. As~$R$ is coarsely embedded in~$G$ by \cref{lem:NonElliptic:DivergingDoubleRay}, there is some $K'$ such that $R_{\leq 0}$ and $R_{\geq K'}$ have distance at least $2K+2$ from each other, which implies that the sets $B_G(R_{\leq 0}, K)$ and $B_G(R_{\geq K'}, K)$ are disjoint and not joined by an edge.
    Hence, every $B_G(R_{\leq 0}, K)$--$B_G(R_{\geq K'}, K)$ path meets either $B_G(r_1Rr_{K'-1}, K)$ or $G - G[R, K]$. But since both $R_{\leq 0}$ and $R_{\geq K'}$ lie in the unique end of $G$, there are infinitely many disjoint such paths, of which at most finitely many can meet the set $B_G(r_1Rr_{K'-1}, K)$, which is finite since $G$ is locally finite. Hence, $G-G[R,K]$ is non-empty for all $K \in \N$.
    
    Thus, there exist vertices $x_1,x_2,\ldots$ such that $d(x_i,R)\geq i$ for all $i\in\N$.
    Let $y_i\in V(R)$ such that $d(x_i,R)=d(x_i,y_i)$.
    Since~$\phi$ acts on~$R$ as a translation, there are only finitely many $\langle\phi\rangle$-orbits on~$R$.
    Hence, there is an infinite subset $I\sub \N$ such that all $y_i$ for $i\in I$ lie in the same $\langle \phi\rangle$-orbit.
    We may assume $1 \in I$.
    For every $i \in I$, let $P_i$ be a shortest $y_i$--$x_i$ path and let $\phi_i\in\langle \phi\rangle$ such that $\phi_i(y_i)=y_1$.
    Then, since $G$ is locally finite, there exists an infinite index set $I_1 \sub I$ such that $\phi_i(P_i)\cap G[y_1,1]$ coincides for all $i\in I_1$, amongst which we again find an infinite index set $I_2 \sub I_1$ such that $\phi_i(P_i)\cap G[y_1,2]$ coincides for all $i\in I_2$ and so on.
    This results in a geodesic ray~$S$ that starts at~$y_1$ and satisfies $d_G(u,R)=d_S(u,y_1)$ for all $u\in V(S)$.
    In particular, $S$ has a tail in $G-B_G(R, K)$ for all $K \in \N$.
    Since $G$ is one-ended, there are infinitely many pairwise disjoint $S$--$R$ paths.
    This shows that the components of $G-B_G(R,K)$ that contain a tail of~$S$ are half-long, and hence the component of $G-B_G(R, L)$ containing a tail of~$S$ is half-thick.    
\end{proof}

\begin{lem}\label{lem:withNonElliptic:main}
    Let $G$ be a one-ended, connected, \lf\ graph with a non-elliptic automorphism. 
    Then there are infinitely many rays $R_0,R_1,\ldots$ such that $d_G(R_i, R_j) \geq \max\{i,j\}$ for all $i \neq j \in \N$.
\end{lem}

\begin{proof}
    Let~$\phi$ be a non-elliptic automorphism of~$G$. By \cref{lem:NonElliptic:ExistDoubleRay}, there exists a $\phi$-invariant double ray~$R$ in~$G$.
    By \cref{lem:withNonElliptic:halfthick}, some component~$C$ of $G-B_G(R,0)$ is half-thick.
    Hence, there exists for every $i\in\N$ a half-long component $C_i\subseteq C$ of $G-B_G(R,i)$.
    We will construct an increasing sequence $(K_i)_{i\in\N}$ of integers $K_i \in \N$ and a sequence $(R_i)_{i\in\N}$ of rays such that $R_i$ lies in $G[R,K_i]-B_G(R,K_{i-1}+i)$.
    Then the rays $R_0,R_1,\ldots$ clearly satisfy $d_G(R_i, R_j) \geq \max\{i,j\}$ for all $i \neq j \in \N$ and hence the assertion follows.

    Let $j\in\N$ and assume that we have constructed all $K_i$ and $R_i$ for $i< j$.
    Since $C_{K_{j-1}+j}$ is half-long, there are infinitely many distinct vertices $x_1,x_2,\ldots$ in~$R$ with $d(x_i,C_{K_{j-1}+j})=K_{j-1}+j+1$.
    As $\langle\phi\rangle$ acts as a translation on~$R$, 
    infinitely many of the $x_i$ lie in the same $\langle\phi\rangle$-orbit.
    Hence, we may assume that all~$x_i$ lie in the same $\langle\phi\rangle$-orbit.
    For every $i\in\N$, let $\phi_i\in\langle\phi\rangle$ such that $\phi_i(x_i)=x_1$ and let $y_i\in V(C_{K_{j-1}+j})$ with $d_G(x_i,y_i)=d_G(x_i,C_{K_{j-1}+j})$.
    Since~$G$ is locally finite, the ball $B_G(x_1, K_{j-1}+j+1)$ is finite, and hence there are distinct $i,\ell$ such that $\phi_i(y_i)=\phi_\ell(y_\ell)$. As $\phi_i(R) = R = \phi_\ell(R)$ since $\phi_i,  \phi_\ell \in \langle\phi\rangle$, and because $\phi_i(y_i) = \phi_\ell(y_\ell)$ and $y_i, y_\ell \in V(C_{K_{j-1}+j})$, it follows that $\phi_i(C_{K_{j-1}+j})=\phi_\ell(C_{K_{j-1}+j})$.
    
    Again because $\phi_i, \phi_\ell \in \langle\phi\rangle$, there is $n\in\Z\sm\{0\}$ such that $\phi_i^{-1}\phi_\ell=\phi^n$.
    Then, $\phi^n(C_{K_{j-1}+j}) = \phi^{-1}_i(\phi_\ell(C_{K_{j-1}+j})) = C_{K_{j-1}+j}$; in particular, $\phi^{kn}(y_1) \in C_{K_{j-1}+j}$ for all $k \in \Z$.
    Thus, $\phi^{kn}(x_1)$ has distance $K_{j-1}+j+1$ to~$C_{K_{j-1}+j}$ for every $k \in \Z$.
    Let $P$ be a $y_1$--$\phi^n(y_1)$ path in $C_{K_{j-1}+j}$.
    Then $\bigcup_{k\in\Z}\phi^{kn}(P)$ is connected and contained in $G[R,K_j]-B_G(R,K_{j-1}+j)$ where $K_j$ is equal to the length of~$P$ plus~$K_{j-1}+j+1$.
    Moreover, $\bigcup_{k\in\Z}\phi^{kn}(P)$ is connected, and it is infinite since $\phi$ is not elliptic. Hence, as $G$ is locally finite, $\bigcup_{k\in\Z}\phi^{kn}(P)$ contains a ray $R_j$ (see e.g.\ \cite[Proposition 8.2.1]{Bibel}), which thus lies in $G[R,K_j]-B_G(R,K_{j-1}+j) \supseteq \bigcup_{k\in\Z}\phi^{kn}(P)$.
\end{proof}

\subsection{All automorphisms of \texorpdfstring{$G$}{G} are elliptic} \label{subsec:Elliptic}

We show that every one-ended, \qt, \lf\ graph~$G$ whose automorphisms are all elliptic contains infinitely many rays that are pairwise far apart. For this, we first show in \cref{lem:onlyElliptic:geodesicsDiverge} below that if an elliptic automorphism $\phi$ maps the first vertex~$r_0$ of a `geodesic' ray~$R$ to another vertex of~$R$ far enough away from~$r_0$, then~$\phi(R)$ has a tail that is far apart from~$R$. 
We then apply this result infinitely often to find infinitely many rays in~$G$ such that, for every $n \in \N$, the $n$-th ray has a tail with distance at least~$n$ from all previously chosen rays. 

A subgraph~$X$ of~$G$ is \defn{geodesic (in~$G$)} if for every two vertices $u, v \in V(X)$ we have $d_X(u, v) = d_G(u, v)$.

\begin{lem}\label{lem:onlyElliptic:geodesicsDiverge}
    Let $K\in\N$, let $G$ be a \lf\ graph, and let $R=r_0r_1\ldots$ be a geodesic ray in~$G$.
    Let~$\phi$ be an elliptic automorphism of~$G$ such that $\phi(r_0)=r_d$ for some $d\geq 4K$.
    Then there exists $j\in\N$ such that $d_G(R_{\geq j}, \phi(R)) \geq K$.
\end{lem}

\begin{proof}
    Let us suppose for a contradiction that $d_G(R_{\geq j},\phi(R))<K$ for all $j\in\N$.
    We show that $d_G(r_0, \varphi^n(r_0))$ tends to infinity for $n \rightarrow \infty$. This then yields a contradiction as $\phi$ is elliptic, and hence the $\phi$-orbit of~$r_0$ is finite.
    
    Set $i_0 := 0$.
    Since $d_G(R_{\geq j}, \phi(R))<K$ for all $j\in\N$, there exist $d=:i_1<i_2<\ldots \in \N$ such that $d_G(r_{i_j},\phi(R))<K$ for all $j\in \N_{\geq 1}$. Set $\ell_j := i_{j+1}-i_j$ for all $j \in \N$. By restricting to a subsequence of~$(i_j)_{j \in \N}$ if necessary, we may assume that $\ell_j\geq d+2K$ for all $j\in \N_{\geq 1}$.
    For $j \in \N_{\geq 1}$, let $y_j$ be in~$\phi(R)$ with $d_G(r_{i_j},\phi(R))=d_G(r_{i_j},y_j)$, and set $\ell_j':=d_G(\phi(r_{i_j}),y_{j+1})$ for all $j\geq 1$.
    In particular, note that $\phi(r_{i_j})$ appears on $\phi(R)$ between $y_j$ and $y_{j+1}$ as $d_G(r_{i_1}, y_j) < d_G(r_{i_1}, \phi(r_{i_j}))$ and $d_G(r_{i_1}, y_{j+1}) > d_G(r_{i_1}, \phi(r_{i_j}))$, which can be easily calculated (by using $\ell_j \geq d+K$ for the second inequality).
    %\[
        %d_G(r_{i_1},y_j)<d_G(r_{i_1},r_{i_j})+K<d_G(r_{i_1},r_{i_j})+d=d_G(r_{i_0},r_{i_j})=d_G(r_{i_1},\phi(r_{i_j}))
    %\]
    %and thus $\phi(r_{i_j})$ appears on $\phi(R)$ after $y_j$, and
    %\[
        %d_G(r_{i_1},y_{j+1})>d_G(r_{i_1},r_{i_{j+1}})-K\geq d_G(r_{i_1},r_{i_j})+\ell_j-K\geq d_G(r_{i_1},r_{i_j})+d=d_G(r_{i_0},r_{i_j})=d_G(r_{i_1},\phi(r_{i_j}))
    %\]
    %and thus $\phi(r_{i_j})$ appears on $\phi(R)$ before~$y_j$.}
    Then
    \begin{equation} \label{eq:Elliptic:1}
        \begin{aligned}
            \ell_j'+d_G(y_j,\phi(r_{i_j})) &= d_G(\phi(r_{i_j}),y_{j+1}) +  d_G(y_j,\phi(r_{i_j}))= d_G(y_j,y_{j+1})\\
            & \leq d_G(y_j, r_{i_j}) + d_G(r_{i_j}, r_{i_{j+1}}) + d_G(r_{i_{j+1}}, y_{j+1}) < K + \ell_j + K =\ell_j+2K,
        \end{aligned}
    \end{equation}
    where we used for the second equality that $\phi(R)$ is geodesic and that $\phi(r_{i_j})$ appears on $\phi(R)$ between $y_j$ and $y_{j+1}$. Since $\phi(r_0) = r_{i_1} = r_d$, we have
    \begin{align*}
        i_j &= d_G(r_{0}, r_{i_j}) = d_G(\phi(r_0), \phi(r_{i_j})) = d_G(r_{i_1},\phi(r_{i_j}))\\
        &\leq d_G(r_{i_1},r_{i_j})+d_G(r_{i_j},y_j)+d_G(y_j,\phi(r_{i_j})) < (i_j-d) +K+d_G(y_j,\phi(r_{i_j})),
    \end{align*}
    where we again used that $\phi(R)$ is geodesic and that $y_j$ appears on $\phi(R)$ before $\phi(r_{i_j})$. Hence, it follows that $d_G(y_j,\phi(r_{i_j}))>d-K$.
    Combining this with \eqref{eq:Elliptic:1} yields
    \[
    \ell_j'<\ell_j+2K-d_G(y_j,\phi(r_{i_j}))<\ell_j+3K-d, 
    \]
    and hence
    \begin{equation} \label{eq:Elliptic:2}
    d_G(\phi(r_{i_j}), r_{i_{j+1}}) \leq d_G(\phi(r_{i_j}), y_{j+1}) + d_G(y_{j+1}, r_{i_{j+1}}) < \ell'_j + K < (\ell_j+3K-d) + K \leq \ell_j 
    \end{equation}
    as $d \geq 4K$.
    For the next sequence of inequalities, we recall that $\phi^0 = \mathrm{id}$; in particular $\phi^0(r_{i_n}) = r_{i_n}$.
    Since $d_G(r_0,\phi^n(r_0))+d_G(\phi^n(r_0),r_{i_n})\geq d_G(r_0,r_{i_n})=i_n$, we conclude
    \begin{align*}
        d_G(r_0,\phi^n(r_0))&\geq i_n-d_G(\phi^n(r_0),r_{i_n})
        \geq i_n - \sum_{j=0}^{n-1} d_G(\phi^{j+1}(r_{i_{n-j-1}}),\phi^j(r_{i_{n-j}}))\\
        \intertext{(where we used the generalised triangle inequality for the second inequality)} 
        &= i_n - \sum_{j=0}^{n-1} d_G(\phi(r_{i_{n-j-1}}), r_{i_{n-j}})
        \geq i_n-\sum_{j=0}^{n-1}(\ell_{n-j-1}-1)\\
        \intertext{(where we used for the first (in)equality that automorphisms preserve distances and \eqref{eq:Elliptic:2} for the second)}
        &= i_n - \sum_{j=0}^{n-1} (i_{n-j}-i_{n-j-1}-1) = i_n - (i_n-i_0 - n) = n.
    \end{align*}
    Since this last expression tends to infinity for $n\to\infty$, this contradicts $\phi$ being elliptic.
\end{proof}

\begin{lem}\label{lem:onlyElliptic:main}
    Let $K\in\N$ and let $G$ be an infinite, connected, \qt, \lf\ graph all of whose automorphisms are elliptic.
    Then there are infinitely many geodesic rays $R_0,R_1,\ldots$ such that $d_G(R_i, R_j) \geq \max\{i,j\}$ for all $i \neq j \in \N$.
\end{lem}

\begin{proof}
    Since $G$ is infinite, connected, and \lf, a standard compactness argument yields a geodesic ray $R=r_0r_1\ldots$ in~$G$. 
    Since $G$ is \qt, some orbit under the automorphism group of~$G$ contains infinitely many vertices of~$R$.
    We may assume that~$r_0$ lies in this orbit.
    Then there is an automorphism~$\phi_1$ of~$G$ such that $\phi_1(r_0)=r_{i_1}$ for some $i_1\geq 4$.
    For general $j$, we choose $\phi_j$ such that $\phi_j(r_{i_{j-1}})= r_{i_j}$ for some $i_j \in \N$ with $i_j-i_{j-1}\geq 4j$.
    Set $R'_0 := R$, and $R'_j := \phi_1^{-1}\circ \ldots \circ \phi_j^{-1}(R_{\geq i_j})$ for all $j \geq 1$.
    We claim that for every $k < j$ the ray~$R'_j$ has a tail~$R^k_j$ that has distance at least $j$ from $R'_{k}$. The assertion then follows by setting $R_j := \bigcap_{k < j} R^k_j$, which is still a ray as $j$ is finite. 

    So let $k < j$ be given. Since all automorphisms of~$G$ are elliptic, also $\phi := \phi_j \circ \ldots \circ \phi_{k+1}$ is elliptic. As also $\phi(r_{i_k}) = r_{i_j}$ and $i_j - i_k \geq i_j - i_{j-1} \geq 4j$, applying \cref{lem:onlyElliptic:geodesicsDiverge} to $R_{\geq i_k}$, $\phi$ and $K := j$ shows that there is some $n \geq i_j$ such that $d_G(R_{\geq n}, \phi(R_{\geq i_k})) \geq j$. 
    Set $R^k_j := \phi_1^{-1}\circ \ldots \circ \phi_j^{-1}(R_{\geq n})$. As $R'_k = \phi_1^{-1}\circ \ldots \circ \phi_k^{-1}(R_{\geq i_k}) = \phi_1^{-1}\circ \ldots \circ \phi_j^{-1}(\phi(R_{\geq i_k}))$ and because automorphisms preserve distances, it follows that $d_G(R_j^k, R'_k) = d_G(R_{\geq n}, \phi(R_{\geq i_k})) \geq j$. Since $R'_j = \phi_1^{-1} \circ \ldots \circ \phi_j^{-1}(R_{\geq i_j})$ and $n \geq i_j$, we have $R_j^k \subseteq R'_j$, and hence $R_j^k$ is as desired.
\end{proof}

\subsection{Proof of \texorpdfstring{\cref{thm:QTOneEnd:HalfGrid}}{Theorem 1}} \label{subsec:ProofOfThm1}

We can now prove \cref{thm:QTOneEnd:HalfGrid}, by combining \cref{lem:onlyElliptic:main,lem:withNonElliptic:main} with \cref{thm:Halin:HalfGrid:UltraFat}.

\begin{proof}[Proof of \cref{thm:QTOneEnd:HalfGrid}.]
    Recall that (the main statement of) \cref{thm:QTOneEnd:HalfGrid} asserts that every one-ended, quasi-transitive, locally finite graph $G$ contains the half-grid as an ultra-fat minor.
    
    If some automorphism of $G$ is non-elliptic, then \cref{lem:withNonElliptic:main} implies that $G$ contains infinitely many rays $R_0,R_1,\ldots$ such that $d_G(R_i, R_j) \geq \max\{i,j\}$ for all $i \neq j \in \N$.
    If all automorphisms of $G$ are elliptic, then we obtain the same by \cref{lem:onlyElliptic:main}.
    Thus, \cref{thm:Halin:HalfGrid}~\ref{itm:Halin:HG:UF:Copy} implies that~$G$ contains an ultra-fat model of the half-grid.

    The `in particular'-part of \cref{thm:QTOneEnd:HalfGrid} (that $G$ then also contains the half-grid as an asymptotic and diverging minor) follows by \cref{obs:UFHGImpliesAsympHG} and \cref{prop:UFHGImpliesDivHG}.
\end{proof}

\section{Applications of \texorpdfstring{\cref{main:HG:QT:OneEnded}}{Theorem 1}} \label{sec:ProofsOfMainResults}

In this section we present some applications of \cref{main:HG:QT:OneEnded}; in particular, we show \cref{main:HalfGrid:QuasiIsoToTree,main:DivergingRays}.

\subsection{Half-grid minors in multi-ended, quasi-transitive graphs}

As a corollary of \cref{main:HG:QT:OneEnded}, we obtain two results on \qt, \lf\ graphs that need not be one-ended.

A finite set $U \subseteq V(G)$ \defn{distinguishes} two ends $\eps, \eps'$ of $G$ if no component of $G-U$ contains rays from both $\eps$ and $\eps'$. 
An end~$\eps$ of~$G$ is \defn{accessible} if there exists some $n \in \N$ such that every other end of~$G$ can be distinguished from~$\eps$ by a set of at most~$n$ vertices of~$G$.
A graph~$G$ is \defn{accessible} if there exists some $n \in \N$ such that every two distinct ends of~$G$ can be distinguished by a set of at most~$n$ vertices of~$G$.

\begin{cor}\label{thm:QTAccess:HalfGrid}
    Let $\eps$ be an accessible, thick end of a \qt, locally finite graph~$G$.
    Then~$G$ contains an ultra-fat model of the half-grid all of whose rays are $\eps$-rays.
\end{cor}

\begin{proof}
    By the same argument as in \cite[Lemma~4.6]{AH24}, there exists a connected, one-ended, \qg,\footnote{A subgraph $X$ of $G$ is \defn{quasi-geodesic} in~$G$ if there exists some $c \in \N$ such that $d_X(u,v) \leq c\cdot d_G(u,v)$ for all $u,v \in V(X)$.} \qt\ subgraph $X$ of~$G$ such that every ray in~$X$ is an $\eps$-ray in~$G$.
    By \cref{thm:QTOneEnd:HalfGrid}, there is an ultra-fat model $(\cV,\cE)$ of the half-grid in~$X$. Since~$X$ is $c$-quasi-geodesic for some $c \in \N$, the submodel of $(\cV, \cE)$ corresponding to $\N^2\sm [cK]^2$ is $K$-fat in~$G$ for all $K \in \N$. It is now straightforward to turn $(\cV, \cE)$ into an ultra-fat model of the half-grid in~$G$.
    In fact, as every ray in~$X$ is an $\eps$-ray in~$G$, we find $HG \prec_{UF}^\eps G$.  
\end{proof}

\begin{cor} \label{cor:HalfGrid:AccessibleGraph}
    Let $G$ be an accessible, \qt, \lf\ graph with a thick end.
    Then the half-grid is an ultra-fat minor of~$G$.
\end{cor}

\begin{proof}
    Since all ends in an accessible, \qt, \lf\ graph are accessible, the claim follows directly from \cref{thm:QTAccess:HalfGrid}.
\end{proof}

For two graphs $G$ and $H$, a map $f\colon V(H) \to V(G)$ is a \defn{quasi-isometry} if there exist $M \geq 1$ and $A \geq 0$ such that
\begin{itemize}
    \item $M^{-1}\cdot d_H(u, v) - A \leq d_G(f(u), f(v)) \leq M\cdot d_H(u, v) + A$ for all $u, v \in V(H)$, and
    \item $d_G(f(V(H)), w) \leq A$ for all $w \in V(G)$.
\end{itemize}
Two graphs are \defn{\qi} if there exists a quasi-isometry between them.

Finally, we can prove \cref{main:HalfGrid:QuasiIsoToTree}.
In fact, we prove the following more detailed version:

\begin{customcor}{\cref*{main:HalfGrid:QuasiIsoToTree}$^\prime$}
\label{cor:HalfGrid:QuasiIsoToTree}
    \emph{Let $G$ be an accessible, connected, \qt, locally finite graph. Then the following are equivalent:
    \begin{enumerate}[label=\rm{(\roman*)}]
        \item \label{itm:QuasiIsoToTree:1} $G$ has a thick end.
        \item \label{itm:QuasiIsoToTree:2} The half-grid is an ultra-fat minor of~$G$.
        \item \label{itm:QuasiIsoToTree:3} The half-grid is an asymptotic minor of~$G$.
        \item \label{itm:QuasiIsoToTree:4} The half-grid is a diverging minor of~$G$.
        \item \label{itm:QuasiIsoToTree:5} $G$ is not quasi-isometric to a tree.
    \end{enumerate}}
\end{customcor}

\begin{proof}
    \ref{itm:QuasiIsoToTree:1} $\Leftrightarrow$ \ref{itm:QuasiIsoToTree:5} is a result by Kr\"on and M\"oller \cite[Theorem~5.5]{KM08}. \ref{itm:QuasiIsoToTree:1} $\Rightarrow$ \ref{itm:QuasiIsoToTree:2} is \cref{cor:HalfGrid:AccessibleGraph}. \ref{itm:QuasiIsoToTree:2} $\Rightarrow$ \ref{itm:QuasiIsoToTree:3} is \cref{obs:UFHGImpliesAsympHG}, and \ref{itm:QuasiIsoToTree:2} $\Rightarrow$ \ref{itm:QuasiIsoToTree:4} is \cref{prop:UFHGImpliesDivHG}. Finally, \ref{itm:QuasiIsoToTree:3} $\Rightarrow$ \ref{itm:QuasiIsoToTree:1} and \ref{itm:QuasiIsoToTree:4} $\Rightarrow$ \ref{itm:QuasiIsoToTree:1} hold because the rays in a model of the half-grid are all equivalent, and thus the end of~$G$ that contains them must be thick.
\end{proof}

We conjecture that \cref{cor:HalfGrid:QuasiIsoToTree} holds for all connected, \qt, \lf\ graphs (even for those that are not accessible). To prove this, it would suffice to show the following:

\begin{conj} \label{conj:QuasiIsoToTree}
    Every connected, \qt, \lf\ graph with a thick end contains an ultra-fat model of the half-grid.
\end{conj}

By \cref{cor:HalfGrid:QuasiIsoToTree} and because every inaccessible, \qt, \lf\ graph has a thick end, it remains to show for a proof of \cref{conj:QuasiIsoToTree} that every inaccessible, \qt, \lf\ graph contains an ultra-fat model of the half-grid.

\subsection{Diverging rays in quasi-transitive graphs}

In this section we apply \cref{main:HG:QT:OneEnded} to prove \cref{main:DivergingRays}, which we restate here for convenience:

\begin{customthm}{\cref*{main:DivergingRays}} \label{thm:DivergingRays}
    \emph{Every \qt, \lf\ graph that contains infinitely many pairwise disjoint rays also contains infinitely many pairwise diverging rays.}
\end{customthm}

Recall that two rays $R,S$ in a graph $G$ \defn{diverge} if for every $K \in \N$ they have tails $R' \subseteq R$, $S' \subseteq S$ satisfying $d_G(R', S') > K$.

\begin{proof}
    Let $G$ be a \qt, \lf\ graph with infinitely many pairwise disjoint rays.
    If $G$ has infinitely many ends, then we may take an infinite set of rays no two of which lie in the same end.
    Indeed, let $R_1$ and~$R_2$ be rays from distinct ends and let $S$ be a finite vertex set such that $R_1$ and~$R_2$ have tails in distinct components $C_1$ and~$C_2$ of $G-S$, respectively.
    Since $G$ is \lf, each $R_i$ has a tail in $C_i-B_G(S,n)$ for every $n\in\N$.
    Thus, $R_1$ and~$R_2$ diverge.

    So, we may assume that $G$ has only finitely many ends.
    Then infinitely many of the pairwise disjoint rays lie in the same end, which hence must be thick. As every graph with only finitely many ends is accessible, \cref{cor:HalfGrid:AccessibleGraph} yields that $G$ contains an ultra-fat model $(\cV, \cE)$ of the half-grid.
    Without loss of generality we may assume that the branch sets are finite.
    For every $i \in \N$, there exists a ray $R_i$ in $G$ that is contained in the union of the branch sets and paths of $(\cV, \cE)$ corresponding to the $i$-th vertical ray $S_i$ of the half-grid. Since all branch sets and paths are finite, $R_i$ has a tail outside of the branch sets and paths corresponding to the first $K$ vertices and edges of $S_i$. Since $(\cV, \cE)$ is ultra-fat, it follows that the rays $R_i$ diverge pairwise.
\end{proof}

\printbibliography

\end{document}